\documentclass[english]{amsart}
\usepackage[T1]{fontenc}
\usepackage[latin9]{inputenc}
\usepackage{geometry}
\usepackage{xcolor}
\usepackage{babel}
\usepackage{float}
\usepackage{units}
\usepackage{mathtools}
\usepackage{amstext}
\usepackage{amsthm}
\usepackage{amssymb}
\usepackage{stmaryrd}
\usepackage{graphicx}
\usepackage[unicode=true,pdfusetitle,
 bookmarks=true,bookmarksnumbered=false,bookmarksopen=false,
 breaklinks=false,pdfborder={0 0 1},backref=false,colorlinks=false]
 {hyperref}

\makeatletter

\providecommand{\tabularnewline}{\\}

\numberwithin{equation}{section}
\numberwithin{figure}{section}
\theoremstyle{plain}
\newtheorem{thm}{\protect\theoremname}[section]
\theoremstyle{definition}
\newtheorem{defn}[thm]{\protect\definitionname}
\theoremstyle{definition}
\newtheorem{example}[thm]{\protect\examplename}
\theoremstyle{plain}
\newtheorem{lem}[thm]{\protect\lemmaname}
\theoremstyle{plain}
\newtheorem{cor}[thm]{\protect\corollaryname}
\theoremstyle{plain}
\newtheorem{prop}[thm]{\protect\propositionname}
\theoremstyle{remark}
\newtheorem{rem}[thm]{\protect\remarkname}
\theoremstyle{remark}
\newtheorem*{rem*}{\protect\remarkname}

\usepackage{subcaption}
\usepackage{graphicx}
\usepackage{xcolor}
\usepackage{tikz}
\usepackage{pgfplots}
\usetikzlibrary{positioning}

\makeatother

\providecommand{\corollaryname}{Corollary}
\providecommand{\definitionname}{Definition}
\providecommand{\examplename}{Example}
\providecommand{\lemmaname}{Lemma}
\providecommand{\propositionname}{Proposition}
\providecommand{\remarkname}{Remark}
\providecommand{\theoremname}{Theorem}

\begin{document}
\global\long\def\R{\mathbb{R}}%

\global\long\def\C{\mathbb{C}}%

\global\long\def\K{\mathbb{K}}%

\global\long\def\Q{\mathbb{Q}}%

\global\long\def\Z{\mathbb{Z}}%

\global\long\def\N{\mathbb{N}}%

\global\long\def\M{\mathcal{M}}%

\global\long\def\U{\mathbb{U}}%

\global\long\def\e{\text{e}}%

\global\long\def\norm#1#2{\left\Vert #1\right\Vert _{#2}}%

\global\long\def\diff{\,d}%

\global\long\def\barre{\,\middle|\,}%

\global\long\def\tendsto#1#2{\underset{#1\to#2}{\longrightarrow}}%

\global\long\def\transpose#1{\prescript{t}{}{#1}}%

\global\long\def\weakto#1#2{\underset{#1\to#2}{\rightharpoonup}}%

\global\long\def\A{\mathcal{A}}%

\global\long\def\X{\mathcal{X}}%

\title{Strong stochastic stability of cellular automata}

\author{Marsan~Hugo}
\address{IMT, Université Toulouse III - Paul Sabatier, Toulouse, France}
\email{hugo.marsan@math.univ-toulouse.fr}

\author{Sablik~Mathieu}
\address{IMT, Université Toulouse III - Paul Sabatier, Toulouse, France}
\email{msablik@math.univ-toulouse.fr}
\urladdr{https://www.math.univ-toulouse.fr/\~msablik/index.html}

\thanks{Mathieu Sablik was partially supported by ANR project Difference (ANR-20-CE48-0002) and the project Computability of asymptotic properties of dynamical systems from CIMI Labex (ANR-11-LABX-0040).}

\begin{abstract}
We define the notion of stochastic stability, already present in the
literature in the context of smooth dynamical systems, for invariant
measures of cellular automata perturbed by a random noise, and the
notion of strongly stochastically stable cellular automaton. We study
these notions on basic examples (nilpotent cellular automata, spreading
symbols) using different methods inspired by those presented in \cite{MST19}.
We then show that this notion of stability is not trivial by proving
that a Turing machine cannot decide if a given invariant measure of
a cellular automaton is stable under a uniform perturbation.
\end{abstract}

\maketitle

\section{Introduction}

\subsection{Stochastic stability: physics motivation}

Dynamical systems, like cellular automata, are models for physical
observations. They can be studied as deterministic models, despite
the presence of errors compared to the real phenomenon: model errors,
measures errors, small perturbation, etc. The study of stochastic
stability (or zero-noise limit) aim to determine on the behaviors
encountered in those deterministic models, which one are resistant
to noise, and thus can be thought of as having a physical ``sense''.

More precisely, let us define a discrete dynamical system $\left(\X,F\right)$
with $\X$ a compact metric system and $F:\X\to\X$ a continuous map.
The long-term behavior can be described by their invariant measures:
denote one of them by $\pi_{0}$. To decide which ones had physical
meaning, A. N. Kolmogorov proposed the following tool \cite{ER85}:
suppose a family $\left(F_{\epsilon}\right)_{\epsilon>0}$ of dynamics
obtained by perturbation of $F$ by a noise of size $\epsilon$. For
each, denote by $\pi_{\epsilon}$ a $F_{\epsilon}$-invariant measure.
If $\pi_{\epsilon}\tendsto{\epsilon}0\pi_{0}$ (in some sense), then
$\pi_{0}$ is said to be \emph{stochastically stable} (or statistically
stable) under small perturbation.

This question is the subject of lot of articles in the context of
smooth dynamical systems, where $\X$ is a Riemannian manifold and
$F$ have some regularity properties (or not), and the measures considered
are often continuous with respect to the Lebesgue measure \cite{ER85,Young02}.
Further works even studied the regularity properties at $0$ of the
map $\epsilon\to\pi_{\epsilon}$ and their link to the speed of convergence
(the linear response \cite{Baladi14} or even quadratic response \cite{GS20}).

\subsection{Cellular automata: computer science (and other) motivation}

Cellular automata (CA) were first introduced by Von Neumann at the
end of the 40's to model local interactions phenomenons \cite{Neumann66}.
A cellular automaton can be defined as a dynamical system defined
by a local rule which acts synchronously and uniformly on the configuration
space $\A^{\Z^{d}}$ where $\A$ is a finite alphabet. These simple
models have a wide variety of different dynamical behaviors and they
are used to model physical systems defined by local rules but also
models of massively parallel computers.

Their perturbed counterpart, Probabilistic Cellular Automata (PCA)
are studied to understand the robustness of their computation, in
particular their dependence towards the initial condition. When a
PCA is \emph{ergodic}, in the sense that every trajectory converges
to the same distribution, it forgets its initial condition, and thus
no reliable computation is possible. PCA are generally ergodic and
in \cite{MST19} the authors exhibit large classes of cellular automata
which have this behavior. There exists some examples of non ergodic
CA in dimension 2 and higher \cite{Toom01}. In dimension 1 non ergodic
CA are more complex \cite{Gacs01} and their construction is based
on fault-tolerant model of computation.

If this problem of fault-tolerant models comes from theoretical computer
science, it could also have practical application. The perturbation
of a cellular automaton can be thought of as errors that can occur
in the update of a computer bit. If such errors are really rare in
our daily computers, they are (theoretically) more frequent in computers
aboard spacecrafts, as they are more vulnerable to cosmic neutron
rays \cite{HH99}. In this paper the errors will occur with probability
$\epsilon>0$.

\subsection{Stochastic stability for cellular automata}

It is natural to try to understand the effect of small random perturbations
on the dynamics of cellular automata and more precisely if the behavior
of the deterministic model can be observed despite the presence of
small errors. As models of computation, they can be used to study
the reliability of computation against noise.

Since a large classes of cellular automata are ergodic, we don't need
the help of other assumptions (like SRB measures) to have the uniqueness
of the invariant measure for each perturbed PCA, which allows us to
study the limit(s) of a family of measures $\left(\pi_{\epsilon}\right)_{\epsilon>0}$
when $\epsilon$ goes to zero. We can hope that this behavior select
only a few of the invariant measures of the deterministic CA, as they
are much more inclined to have a lot of invariant measures. If the
CA is not ergodic, we consider stable measures as the set of adherence
values when $\epsilon$ goes to $0$ of invariant measures of the
perturbed system.

This notion of stability is quite similar to the stability of trajectories
studied in \cite{GT22}. In this article the authors characterize
monotonic cellular automata such that the orbit of the trajectory
of the uniform configuration with the symbol $0$ stays near this
configuration when the perturbation parameter goes to 0. When the
cellular automata is ergodic, this notion implies the previous notion
of stochastic stability for the Dirac mass on the configuration with
only $0$s. Another notion of stability also appears in \cite{FMT22}
to study how a probabilistic cellular automaton can correct mistakes
of some tilings defined by local rules.

The first examples of CA that would seem stable are classes of CA
which converges rapidly to a fixed point: we take the example of nilpotent
CA. Another interesting case would be classes of CA with several fixed
points, but that all but one could be described as ``unstable''.
Here, we take the example of CA where a symbol is spreading, and verify
that the stochastic stability only select the ``stable'' point (with
no necessary the monotonic assumption as in \cite{FMT22}). If the
notion of stochastic stability is intuitive, proving that a particular
CA is stable may not be. In fact, we prove that it is an undecidable
property; we can draw parallel to other ``basic'' properties that
are in fact undecidable for CA, like nilpotency \cite{Kari92}.

\subsection{Description of the paper}

In section 2 we recall the basic tools for the study of cellular automata,
and define the ones for the study of their stochastic stability nature.

In section 3 and 4, we apply this notion on simple examples where
we expect stability to appear: nilpotent CA and CA with a spreading
symbol. The stable measure for those automata is very simple, as it
is the Dirac mass on a uniform configuration. Beside proving the stability
of this measure, we also show different approaches to obtain an upper
bound on the speed of convergence towards it.

The two final sections are independent of the two previous ones. In
section 5 we present proofs for computation results we use in the
following section. Those results are Proposition \ref{prop:equivalentaubord}
and \ref{prop:limiteaubord}, which gives an asymptotic development
for several functions when the noise goes to zero.\textcolor{red}{{}
}Finally, in section 6 we prove that given a CA perturbed by a standard
noise, the stochastic stability of a measure is undecidable, as stated
in Theorem \ref{thm:undecidabilite}. To prove this theorem, we simulate
a Turing machine in a construction already described in \cite{BDPST15}
and \cite{ENT22}.

\section{Stochastic stability for cellular automata}

Let $\A$ be a finite alphabet of symbols, and define $\X=\A^{\Z^{d}}$
the space of configurations of $\Z^{d}$ endowed with the product
topology. An application $F:\X\to\X$ is a cellular automaton (CA)
if there is a finite neighborhood $\mathcal{N}=\left\{ i_{1},...,i_{r}\right\} \subset\Z^{d}$
and a local rule $f:\A^{\mathcal{N}}\to\A$ such that for all $i\in\Z^{d}$,$\left(Fx\right)_{i}=f(x_{i+\mathcal{N}})$
where $x_{i+\mathcal{N}}=(x_{i+i_{1}},x_{i+i_{2}},...,x_{i+i_{r}})$. 

A transition kernel $\Phi$ is a probabilistic cellular automaton
(PCA) if there is a finite neighborhood $\mathcal{N}=\left\{ i_{1},...,i_{r}\right\} \subset\Z^{d}$
and a stochastic matrix (local rule) $\varphi:\A^{N}\times\A\to[0,1]$
such that for all $x\in\A^{\Z^{d}}$, $\Phi(x,[u]_{A})=\prod_{i\in A}\varphi(x_{i+\mathcal{N}},u_{i})$.
Moreover, it is a $\epsilon$-perturbation of a CA $F$ if they are
defined on the same alphabet, have the same neighborhood, and if their
local rules $\varphi$ and $f$ verify for all $a_{1},a_{2},...,a_{r}\in\A$,
$\varphi\left(\left(a_{1},...,a_{r}\right),f(a_{1},...,a_{r})\right)\geq1-\epsilon.$

Deterministic and probabilistic cellular automata both acts on $\M(\X)$
the set of Borel probability measures on $\X$, by $\Phi\mu(A)=\int\Phi(x,A)\diff\mu(x)$
for any $\Phi$ PCA, $\mu\in\M(\X)$ and $A$ observable. A measure
$\mu$ is $\Phi$-invariant if $\Phi\mu=\mu$. Recall that $\M(\X)$
is compact and metrizable for the weak convergence topology: $\mu_{n}\weakto n{\infty}\mu$
if $\mu_{n}([u]_{A})\tendsto n{\infty}\mu([u]_{A})$ for all cylinders
$[u]_{A}$.
\begin{defn}[Stochastic stability of a measure]
A measure $\pi\in\M(X)$ is \emph{stochastically stable} under $\left(F_{\epsilon}\right)_{\epsilon>0}$
if there exists a numerical sequence $\left(\epsilon_{n}\right)_{n\in\N}$
converging towards $0$ and a sequence $\left(\pi_{\epsilon_{n}}\right)_{n\in\N}$
verifying:
\begin{enumerate}
\item $\forall n\in\N,\pi_{\epsilon_{n}}$ is $F_{\epsilon_{n}}$-invariant.
\item $\pi_{\epsilon_{n}}\weakto n{\infty}\pi$.
\end{enumerate}
\end{defn}

\begin{defn}[Strong stochastic stability of a cellular automaton]
A cellular automaton $F$ is \emph{strongly stochastically stable}
under $\left(F_{\epsilon}\right)_{\epsilon>0}$ if it admits only
one stochastically stable measure.
\end{defn}

Observe that by definition of an $\epsilon$-perturbation and continuity
of the action of $F_{\epsilon}$ on $\M(\X)$, all stochastically
stable measures are invariant measure for $F$. 

In order to compare speeds of convergence, we use the total variation
distance on a finite observation window: for a finite set $A\subset\Z^{d}$
and two measures $\mu,\nu\in\M(\X)$, define $\norm{\mu-\nu}A\coloneqq\frac{1}{2}\sum_{u\in\A^{A}}\left|\mu\left(\left[u\right]\right)-\nu\left(\left[u\right]\right)\right|$
. If $\left(\mu_{n}\right)_{n}$ is a sequence of $\M(\X)$, the following
equivalence holds: 
\[
\mu_{n}\weakto n{\infty}\mu\Leftrightarrow\forall A\subset\Z^{d}\text{ finite,}\norm{\mu_{n}-\mu}A\tendsto n{\infty}0.
\]

Finally, for a symbol $a\in\A$, we denote by $a^{\infty}$ the configuration
$x\in\A^{\Z^{d}}$ such that $x_{i}=a$ for all $i\in\Z^{d}$. The
Dirac mass concentrated on this configuration will be denoted by $\delta_{a}$.

\section{Nilpotent CA}

The first class of CA we can study is the nilpotent ones. A cellular
automaton $F$ is said to be \emph{nilpotent} if there is a integer
$N\in\N^{*}$ such that $F^{N}$ is a constant function. By shift-invariance
a nilpotent CA admits a symbol, which we will denote by $0\in\A$,
such that $F^{N}$ is the constant function equals to $0^{\infty}$.
As they only admit $\delta_{0}$ as an invariant measure, it is immediate
that it is also stochastically stable. The authors of \cite{MST19}
prove that for small perturbations, ergodicity is conserved. We can
reuse the same kind of arguments to prove an upper bound on the speed
of convergence on a finite window of observation.
\begin{thm}[Stability for nilpotent CA]
\label{thm:Nilpotents}Let $\left(F_{\epsilon}\right)_{\epsilon>0}$
a family of $\epsilon$-perturbations of a nilpotent CA $F$ on $\Z^{d}$.
For $\epsilon$ small enough, we denote by $\pi_{\epsilon}$ the unique
invariant measure of $F_{\epsilon}$. Then there is a constant $C>0$
such that for all finite $A\subset\Z^{d}$, 
\[
\norm{\delta_{0}-\pi_{\epsilon}}A\leq1-\left(1-\epsilon\right)^{\left|A\right|C}\leq C\left|A\right|\epsilon.
\]
\end{thm}

\begin{proof}
Denote by $[0]_{A}$ the cylinder $\left\{ x\in\mathcal{X}\mid\forall i\in A,x_{i}=0\right\} $.
One easily gets 
\begin{align*}
\norm{\delta_{0}-\pi_{\epsilon}}A & =1-\pi_{\epsilon}\left(\left[0\right]_{A}\right).
\end{align*}
Using \cite{MST19}'s notations, we denote by $\mathcal{N}^{t}$ the
neighborhood of the CA $F^{t}$, and $m_{t}\coloneqq\left|\mathcal{N}^{t}\right|$
(with $m_{0}\coloneqq1$). By definition of an $\epsilon$-perturbation,
one has $F_{\epsilon}(x,\left[Fx\right]_{A})\geq(1-\epsilon)^{\left|A\right|}$
(i.e. there is no mistake in each cell of $A$). By iterating it,
one gets 
\[
F_{\epsilon}^{N}(x,\underset{=[0]_{A}}{\underbrace{\left[F^{N}x\right]_{A}}})\geq\left(\prod_{t=0}^{N-1}(1-\epsilon)^{m_{t}}\right)^{|A|}=(1-\epsilon)^{|A|\cdot\sum_{t=0}^{N-1}m_{t}}
\]
 (i.e. for each points of $A$, there is no mistake in its neighborhood
for the last $N$ iterations, i.e. on the points inside the dotted
area on Figure \ref{fig:Nilpotents}). 

\begin{figure}[h]
\begin{centering}
\def\svgwidth{\columnwidth * 4/5}
\begingroup%
  \makeatletter%
  \providecommand\color[2][]{%
    \errmessage{(Inkscape) Color is used for the text in Inkscape, but the package 'color.sty' is not loaded}%
    \renewcommand\color[2][]{}%
  }%
  \providecommand\transparent[1]{%
    \errmessage{(Inkscape) Transparency is used (non-zero) for the text in Inkscape, but the package 'transparent.sty' is not loaded}%
    \renewcommand\transparent[1]{}%
  }%
  \providecommand\rotatebox[2]{#2}%
  \newcommand*\fsize{\dimexpr\f@size pt\relax}%
  \newcommand*\lineheight[1]{\fontsize{\fsize}{#1\fsize}\selectfont}%
  \ifx\svgwidth\undefined%
    \setlength{\unitlength}{568.30125704bp}%
    \ifx\svgscale\undefined%
      \relax%
    \else%
      \setlength{\unitlength}{\unitlength * \real{\svgscale}}%
    \fi%
  \else%
    \setlength{\unitlength}{\svgwidth}%
  \fi%
  \global\let\svgwidth\undefined%
  \global\let\svgscale\undefined%
  \makeatother%
  \begin{picture}(1,0.4729036)%
    \lineheight{1}%
    \setlength\tabcolsep{0pt}%
    \put(0,0){\includegraphics[width=\unitlength,page=1]{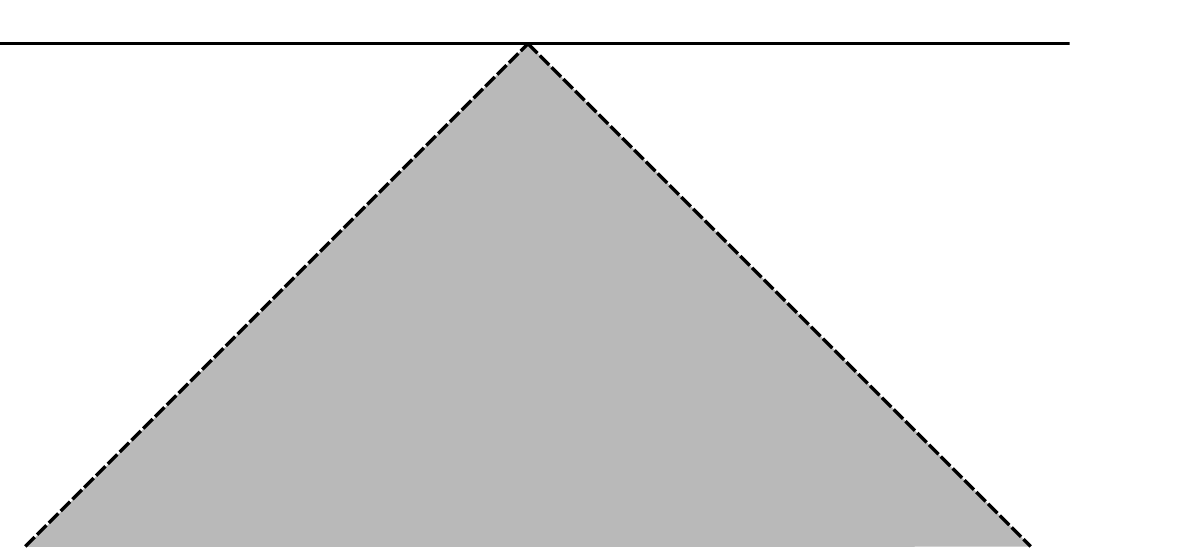}}%
    \put(0.94524522,0.42894726){\color[rgb]{0,0,0}\makebox(0,0)[t]{\lineheight{1.25}\smash{\begin{tabular}[t]{c}$t=0$\\\end{tabular}}}}%
    \put(0.44651188,0.4464076){\color[rgb]{0,0,0}\makebox(0,0)[t]{\lineheight{1.25}\smash{\begin{tabular}[t]{c}$0$\end{tabular}}}}%
    \put(0,0){\includegraphics[width=\unitlength,page=2]{nilpotent.pdf}}%
    \put(0.9523732,0.00406343){\color[rgb]{0,0,0}\makebox(0,0)[t]{\lineheight{1.25}\smash{\begin{tabular}[t]{c}$t=-N$\\\end{tabular}}}}%
    \put(0,0){\includegraphics[width=\unitlength,page=3]{nilpotent.pdf}}%
    \put(0.44568871,0.20254149){\color[rgb]{0,0.44313725,0}\makebox(0,0)[t]{\lineheight{1.25}\smash{\begin{tabular}[t]{c}$m_t$\end{tabular}}}}%
  \end{picture}%
\endgroup%

\par\end{centering}
\caption{\label{fig:Nilpotents}Proof of theorem \ref{thm:Nilpotents}. To
have a $0$ at $t=0$, it suffices to not make any mistake on the
cells inside the grayed area.}
\end{figure}

By $\pi_{\epsilon}$-invariance, one gets
\begin{align*}
\pi_{\epsilon}([0]_{A}) & =F_{\epsilon}^{N}\pi_{\epsilon}([0]_{A})\\
 & =\int F_{\epsilon}^{N}(x,[0]_{A})\diff\pi_{\epsilon}(x)\\
 & \geq(1-\epsilon)^{|A|\cdot\sum_{t=0}^{N-1}m_{t}}
\end{align*}
and then our result with $C\coloneqq\sum_{t=0}^{N-1}m_{t}$ (independent
of $A$).
\end{proof}

\section{Spreading CA}

A cellular automaton $F$ admits $0\in\A$ as a \emph{spreading state}
if it verifies for all $i\in\Z^{d}$ one has: $$\left[\exists j\in\mathcal{N},x_{i+j}=0\right]\Rightarrow(Fx)_{i}=0.$$

\begin{example}
The CA $F(x)_{i}=x_{i}\cdot x_{i+1}$ defined on $\left\{ -1,0,1\right\} ^{\Z}$
admits $0$ as a spreading state.
\end{example}

Contrary to a nilpotent CA, such an automaton can have several fixed
points: in particular, $\delta_{0}$ is not necessarily the only invariant
measure. However, it is very intuitive to think that, as long as they
can appear, the $0$ will spread on the grid, and the measure we can
observe are thus near $\delta_{0}$. For that reason, we consider
perturbations that are $0$-positive: for any $a_{1},a_{2},...,a_{r}\in\A$,
$\varphi\left(a_{1},...,a_{r}\right)(0)>0$.

If $\A$ is endowed with an order (e.g. $\A=\left\{ 0<1<...<n\right\} $),
such CA can be thought as having similar properties as \emph{monotonous
eroders}, as defined and studied in \cite{Toom80,GT22}. In those
articles, the author studied the stability of the trajectory beginning
at $x=0^{\infty}$. The monotonous eroders CA having this trajectory
stable are called \emph{stable}. It is easy to prove that, if generalizing
this definition of stability to all CA, a stable CA which is ergodic
when perturbed admits $\delta_{0}$ as its unique stochastically stable
measure.

We present two different approaches for different cases: in the first
one, we prove the stochastic stability of $\delta_{0}$ under any
$0$-positive perturbation for 1-dimensional CA admitting $0$ as
a spreading state. As we do not only consider monotonous CA, this
is not an application of \cite{GT22}. In the second one, we prove
the stability for any dimension, but only for a binary alphabet $\A=\left\{ 0,1\right\} $
and a more restrictive class of noise. Here, all spreading CA on a
binary alphabet are monotonous, and thus the stability of $\delta_{0}$
is a consequence of the results in \cite{Toom80}. However, the proof
we propose is based on the computations of \cite{MST19}, which also
provide a speed of convergence in certain cases. 

\subsection{1-dimensional}

In this part, we only consider 1-dimensional CA, i.e. defined on $\A^{\Z}$.

\begin{thm}
\label{thm:Spreading-general}Let $F$ be a CA on $\X=\A^{\Z}$ with
neighborhood $\mathcal{N}$ an interval of $\Z$ with length $\left|\mathcal{N}\right|=r$
admitting $0$ as a spreading state, and $\left(F_{\epsilon}\right)_{\epsilon>0}$
a family of $0$-positive $\epsilon$-perturbations. For all $\epsilon>0$,
let $\pi_{\epsilon}\in\M_{\epsilon}$. Then for all finite interval
$A\subset\Z$ , there is a constant $C_{\left|A\right|}$ such that
\[
\norm{\delta_{0}-\pi_{\epsilon}}A\leq1-\left(1-\epsilon\right)^{C_{\left|A\right|}}\cdot\left(1-\frac{27\epsilon}{1-27\epsilon}\right)\underset{\epsilon\to0}{\sim}\left(C_{\left|A\right|}+27\right)\epsilon.
\]
\end{thm}

\subsubsection{Spread graph}

The following paragraph is adapted from the ideas one can read in
more details in \cite{Toom01} and \cite{FT08}. Let us describe what
is a \emph{spread graph}.\textcolor{red}{{} }We construct it in three
steps, illustrated in Figure \ref{fig:ConstructionGraphe}:
\begin{enumerate}
\item Consider the (infinite) dependency graph of the cell $O$, at position
$0$ and time $t=0$, for a CA with neighborhood $\mathcal{N}=\left\{ 0,...,r-1\right\} $,
tilted to keep symmetry. 
\item In order to use tools for planar graph, each step of the CA is decomposed
into $r-1$ steps of a CA with neighborhood $\left\{ 0,1\right\} $.
Its definition does not matter as we are only considering the spread
of the symbol $0$.
\item To represent the noise, each vertex corresponding to a ``true cell''
at time $t\in\Z^{-}$ and position $i\in\N$ is split into two vertices,
linked with an edge $e(t,i)$. They are always open in the \emph{down}
direction, but are only open in the \emph{up} direction with a probability
greater than $1-\epsilon$, when there is no error in the cell $i$
at time $t$.
\end{enumerate}
\begin{figure}
\begin{subfigure}[b]{0.32\textwidth}         \resizebox{\linewidth}{!}{             \begin{tikzpicture}[ err/.style={circle, draw=red!60, very thick, minimum size=3mm}, cell/.style={circle, draw=blue!60, very thick, minimum size=3mm}, carre/.style={rectangle, draw=red!60, very thick, minimum size=3mm}, ]         \node[cell] at (5,6) (2 0) {$O$};         \node[cell] at (1,0) (0 0) {};         \node[cell] at (3,0) (0 1) {};         \node[cell] at (5,0) (0 2) {};         \node[cell] at (7,0) (0 3) {};         \node[cell] at (9,0) (0 4) {};         \node[cell] at (3,3) (1 0) {};         \node[cell] at (5,3) (1 1) {};         \node[cell] at (7,3) (1 2) {};
        \draw[->] (0 0) -- (1 0);         \draw[->] (0 1) -- (1 0);         \draw[->] (0 2) -- (1 0);         \draw[->] (0 1) -- (1 1);         \draw[->] (0 2) -- (1 1);         \draw[->] (0 3) -- (1 1);         \draw[->] (0 2) -- (1 2);         \draw[->] (0 3) -- (1 2);         \draw[->] (0 4) -- (1 2);         \draw[->] (1 0) -- (2 0);         \draw[->] (1 1) -- (2 0);         \draw[->] (1 2) -- (2 0);                  \end{tikzpicture}         }     \end{subfigure}     \begin{subfigure}[b]{0.32\textwidth}         \resizebox{\linewidth}{!}{             \begin{tikzpicture}[ err/.style={circle, draw=red!60, very thick, minimum size=3mm}, cell/.style={circle, draw=blue!60, very thick, minimum size=3mm}, carre/.style={rectangle, draw=black!60, very thick, minimum size=3mm}, ]         \node[cell] at (1,0) (0 0) {};         \node[cell] at (3,0) (0 1) {};         \node[cell] at (5,0) (0 2) {};         \node[cell] at (7,0) (0 3) {};         \node[cell] at (9,0) (0 4) {};                  \node[carre] at (2,1) (0 01) {};         \node[carre] at (4,1) (0 12) {};         \node[carre] at (6,1) (0 23) {};         \node[carre] at (8,1) (0 34) {};                  \node[cell] at (3,3) (1 0) {};         \node[cell] at (5,3) (1 1) {};         \node[cell] at (7,3) (1 2) {};                  \node[carre] at (4,4) (1 01) {};         \node[carre] at (6,4) (1 12) {};                  \node[cell] at (5,6) (2 0) {$O$};                  \draw[->] (0 0) -- (0 01);         \draw[->] (0 1) -- (0 01);         \draw[->] (0 1) -- (0 12);         \draw[->] (0 2) -- (0 12);         \draw[->] (0 2) -- (0 23);         \draw[->] (0 3) -- (0 23);         \draw[->] (0 3) -- (0 34);         \draw[->] (0 4) -- (0 34);                  \draw[->] (0 01) -- (1 0);         \draw[->] (0 12) -- (1 0);         \draw[->] (0 12) -- (1 1);         \draw[->] (0 23) -- (1 1);         \draw[->] (0 23) -- (1 2);         \draw[->] (0 34) -- (1 2);                  \draw[->] (1 0) -- (1 01);         \draw[->] (1 1) -- (1 01);         \draw[->] (1 1) -- (1 12);         \draw[->] (1 2) -- (1 12);                  \draw[->] (1 01) -- (2 0);         \draw[->] (1 12) -- (2 0);         \end{tikzpicture}         }     \end{subfigure}     \begin{subfigure}[b]{0.32\textwidth}         \resizebox{\linewidth}{!}{             \begin{tikzpicture}[ err/.style={circle, draw=red!60, very thick, minimum size=3mm}, cell/.style={circle, draw=blue!60, very thick, minimum size=3mm}, carre/.style={rectangle, draw=black!60, very thick, minimum size=3mm}, ]         \node[cell] at (1,0) (0 0) {};         \node[cell] at (3,0) (0 1) {};         \node[cell] at (5,0) (0 2) {};         \node[cell] at (7,0) (0 3) {};         \node[cell] at (9,0) (0 4) {};                  \node[carre] at (2,1) (0 01) {};         \node[carre] at (4,1) (0 12) {};         \node[carre] at (6,1) (0 23) {};         \node[carre] at (8,1) (0 34) {};                  \node[err] at (3,2) (1 0p) {};         \node[err] at (5,2) (1 1p) {};         \node[err] at (7,2) (1 2p) {};                  \node[cell] at (3,3) (1 0) {};         \node[cell] at (5,3) (1 1) {};         \node[cell] at (7,3) (1 2) {};                  \node[carre] at (4,4) (1 01) {};         \node[carre] at (6,4) (1 12) {};                  \node[err] at (5,5) (2 0p) {};                  \node[cell] at (5,6) (2 0) {$O$};                  \draw[->] (0 0) -- (0 01);         \draw[->] (0 1) -- (0 01);         \draw[->] (0 1) -- (0 12);         \draw[->] (0 2) -- (0 12);         \draw[->] (0 2) -- (0 23);         \draw[->] (0 3) -- (0 23);         \draw[->] (0 3) -- (0 34);         \draw[->] (0 4) -- (0 34);                  \draw[->] (0 01) -- (1 0p);         \draw[->] (0 12) -- (1 0p);         \draw[->] (0 12) -- (1 1p);         \draw[->] (0 23) -- (1 1p);         \draw[->] (0 23) -- (1 2p);         \draw[->] (0 34) -- (1 2p);                  \draw[->, dashed, color = red] (1 0p) -- (1 0);         \draw[->, dashed, color = red] (1 1p) -- (1 1);         \draw[->, dashed, color = red] (1 2p) -- (1 2);                  \draw[->] (1 0) -- (1 01);         \draw[->] (1 1) -- (1 01);         \draw[->] (1 1) -- (1 12);         \draw[->] (1 2) -- (1 12);                  \draw[->] (1 01) -- (2 0p);         \draw[->] (1 12) -- (2 0p);                  \draw[->, dashed, color = red] (2 0p) -- (2 0);                  \end{tikzpicture}  }     \end{subfigure}

\caption{\label{fig:ConstructionGraphe}Left: step 1, the first three levels
of the initial dependency graph for $r=3$.\protect \\
Center: step 2, the decomposition into a planar graph.\protect \\
Right: step 3, adding the noise.}
\end{figure}
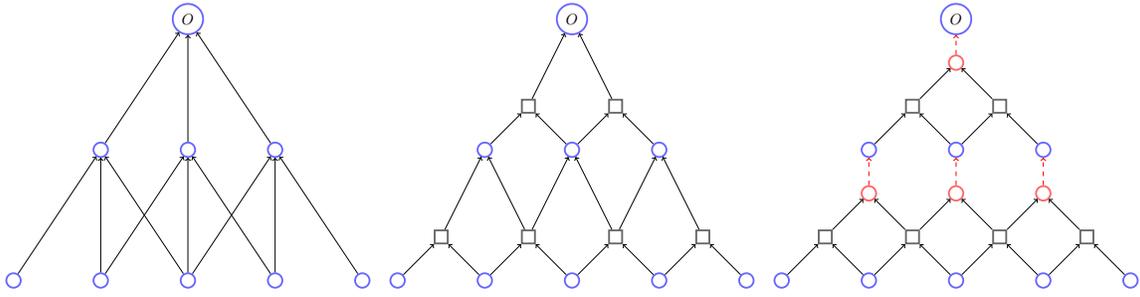

\begin{defn}
The spread graph associated to $\left(U_{i}^{-t}\right)_{i,t\in\Z}=\left(U^{-t}\right)_{t\in\Z}$,
a collection of independent uniform variables on $\left[0,1\right]$,
is the spread graph where the edge $e(-t,i)$ is open in the \emph{up}
direction when $U_{i}^{-t}>\epsilon$.
\end{defn}

The tilted edges, that represent the spread of the symbol $0$ by
the deterministic cellular automaton, are always open in the \emph{up}
direction and closed in the \emph{down} direction. What is the probability
to have a infinite open path ending the top vertex $O$ ? To answer
it, consider the dual of this planar graph as in figure \ref{fig:SpreadingDual}
\begin{figure}
\resizebox{\linewidth}{!}{     \begin{tikzpicture}[ err/.style={circle, draw=red!60, very thick, minimum size=3mm}, cell/.style={circle, draw=blue!60, very thick, minimum size=3mm}, carre/.style={rectangle, draw=black!60, very thick, minimum size=3mm}, ]         \node[cell] at (1,0) (0 0) {};         \node[cell] at (3,0) (0 1) {};         \node[cell] at (5,0) (0 2) {};         \node[cell] at (7,0) (0 3) {};         \node[cell] at (9,0) (0 4) {};                  \node[carre] at (2,1) (0 01) {};         \node[carre] at (4,1) (0 12) {};         \node[carre] at (6,1) (0 23) {};         \node[carre] at (8,1) (0 34) {};                  \node[err] at (3,2) (1 0p) {};         \node[err] at (5,2) (1 1p) {};         \node[err] at (7,2) (1 2p) {};                  \node[cell] at (3,3) (1 0) {};         \node[cell] at (5,3) (1 1) {};         \node[cell] at (7,3) (1 2) {};                  \node[carre] at (4,4) (1 01) {};         \node[carre] at (6,4) (1 12) {};                  \node[err] at (5,5) (2 0p) {};                  \node[cell] at (5,6) (2 0) {$O$};                  \draw[->] (0 0) -- (0 01);         \draw[->] (0 1) -- (0 01);         \draw[->] (0 1) -- (0 12);         \draw[->] (0 2) -- (0 12);         \draw[->] (0 2) -- (0 23);         \draw[->] (0 3) -- (0 23);         \draw[->] (0 3) -- (0 34);         \draw[->] (0 4) -- (0 34);                  \draw[->] (0 01) -- (1 0p);         \draw[->] (0 12) -- (1 0p);         \draw[->] (0 12) -- (1 1p);         \draw[->] (0 23) -- (1 1p);         \draw[->] (0 23) -- (1 2p);         \draw[->] (0 34) -- (1 2p);                  \draw[->, dashed, color = red] (1 0p) -- (1 0);         \draw[->, dashed, color = red] (1 1p) -- (1 1);         \draw[->, dashed, color = red] (1 2p) -- (1 2);                  \draw[->] (1 0) -- (1 01);         \draw[->] (1 1) -- (1 01);         \draw[->] (1 1) -- (1 12);         \draw[->] (1 2) -- (1 12);                  \draw[->] (1 01) -- (2 0p);         \draw[->] (1 12) -- (2 0p);                  \draw[->, dashed, color = red] (2 0p) -- (2 0);
        \node[cell] at (10,-0.5) (d0 0) {};         \node[cell] at (12,-0.5) (d0 1) {};         \node[cell] at (14,-0.5) (d0 2) {};         \node[cell] at (16,-0.5) (d0 3) {};         \node[cell] at (18,-0.5) (d0 4) {};         \node[cell] at (20,-0.5) (d0 5) {};                  \node[cell] at (11,1) (d1 0) {};         \node[cell] at (13,1) (d1 1) {};         \node[cell] at (15,1) (d1 2) {};         \node[cell] at (17,1) (d1 3) {};         \node[cell] at (19,1) (d1 4) {};                  \node[cell] at (12,2.5) (d2 0) {};         \node[cell] at (14,2.5) (d2 1) {};         \node[cell] at (16,2.5) (d2 2) {};         \node[cell] at (18,2.5) (d2 3) {};                  \node[cell] at (13,4) (d3 0) {};         \node[cell] at (15,4) (d3 1) {};         \node[cell] at (17,4) (d3 2) {};                  \node[cell] at (14,5.5) (d4 0) {};         \node[cell] at (16,5.5) (d4 1) {};                  \node[cell] at (15,7) (d5 0) {};         \node at (15,6)  {$O$};        
        \draw[->, dashed, color = red] (d0 0) -- (d0 1);         \draw[->, dashed, color = red] (d0 1) -- (d0 2);         \draw[->, dashed, color = red] (d0 2) -- (d0 3);         \draw[->, dashed, color = red] (d0 3) -- (d0 4);         \draw[->, dashed, color = red] (d0 4) -- (d0 5);                  \draw[->, dashed, color = red] (d2 0) -- (d2 1);         \draw[->, dashed, color = red] (d2 1) -- (d2 2);         \draw[->, dashed, color = red] (d2 2) -- (d2 3);                  \draw[->, dashed, color = red] (d4 0) -- (d4 1);
        \draw[->] (d5 0) -- (d4 0);         \draw[->] (d4 0) -- (d3 0);         \draw[->] (d3 0) -- (d2 0);         \draw[->] (d2 0) -- (d1 0);         \draw[->] (d1 0) -- (d0 0);                  \draw[->] (d4 1) -- (d3 1);         \draw[->] (d3 1) -- (d2 1);         \draw[->] (d2 1) -- (d1 1);         \draw[->] (d1 1) -- (d0 1);                  \draw[->] (d3 2) -- (d2 2);         \draw[->] (d2 2) -- (d1 2);         \draw[->] (d1 2) -- (d0 2);                  \draw[->] (d2 3) -- (d1 3);         \draw[->] (d1 3) -- (d0 3);                  \draw[->] (d1 4) -- (d0 4);                  \draw[->] (d0 5) -- (d1 4);         \draw[->] (d1 4) -- (d2 3);         \draw[->] (d2 3) -- (d3 2);         \draw[->] (d3 2) -- (d4 1);         \draw[->] (d4 1) -- (d5 0);                  \draw[->] (d0 4) -- (d1 3);         \draw[->] (d1 3) -- (d2 2);         \draw[->] (d2 2) -- (d3 1);         \draw[->] (d3 1) -- (d4 0);                  \draw[->] (d0 3) -- (d1 2);         \draw[->] (d1 2) -- (d2 1);         \draw[->] (d2 1) -- (d3 0);                  \draw[->] (d0 2) -- (d1 1);         \draw[->] (d1 1) -- (d2 0);                  \draw[->] (d0 1) -- (d1 0);                       \end{tikzpicture}     }

\caption{\label{fig:SpreadingDual}Left: The first three levels of the original
graph. Right: the first three levels of the the dual graph. Note that
the outer vertices actually represent the same region of the original
graph.}
\end{figure}
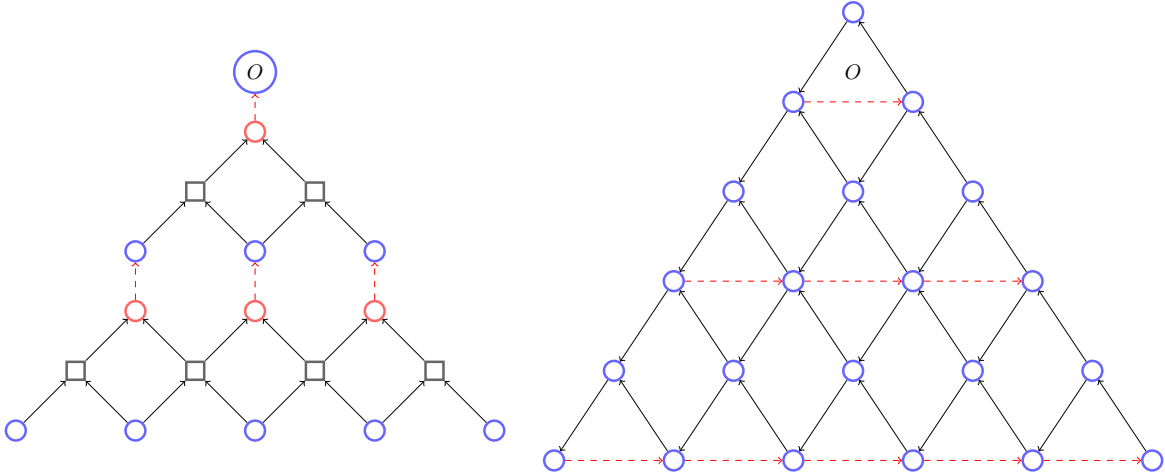
(for a complete definition see \cite{FT08}). Each edge $e$ has a
dual edge $e^{\prime}$. For each direction of $e$, the corresponding
direction on $e^{\prime}$ is the direction from left to right when
we go along $e$ in the given direction. Every edge $e^{\prime}$
is open in a direction if and only if $e$ is closed in the corresponding
direction. For our graph, the following table gives the results. 

\begin{table}[H]
\centering{}%
\begin{tabular}{|c|c|c|c|}
\hline 
 & Original graph & Dual graph & \tabularnewline
\hline 
\hline 
\textcolor{red}{$\uparrow$} & Probability $\geq1-\epsilon$ to be open & Probability $\leq\epsilon$ to be open & \textcolor{red}{$\rightarrow$}\tabularnewline
\hline 
$\downarrow$ & Always open & Always closed & $\leftarrow$\tabularnewline
\hline 
$\nearrow$ & Always open & Always closed & $\searrow$\tabularnewline
\hline 
$\swarrow$ & Always closed & Always open & $\nwarrow$\tabularnewline
\hline 
$\nwarrow$ & Always open & Always closed & $\nearrow$\tabularnewline
\hline 
$\searrow$ & Always closed & Always open & $\swarrow$\tabularnewline
\hline 
\end{tabular}
\end{table}

\begin{lem}[{\cite[Main lemma]{FT08}}]
There is an (infinite) open path ending in $O$ if and only if there
is no open self-avoiding contour in the dual graph leaving $O$ on
the left. 
\end{lem}

\begin{cor}
The probability to have an infinite path ending in $O$ is greater
than $1-\frac{27\epsilon}{1-27\epsilon}$.
\end{cor}

\begin{proof}
Let us bound the probability of existence of an open self-avoiding
contour in the dual graph. We can suppose that every contour begins
and finishes at the top cell of the dual graph. The probability that
there is such a contour is less than $\sum_{k=1}^{\infty}C_{k}\epsilon^{k}$
where $C_{k}$ is the number of contours going through $k$ horizontal
arrows\textcolor{red}{{} $\rightarrow$}. As a contour goes through
an equal number of\textcolor{red}{{} $\rightarrow$}, $\swarrow$ and
$\nwarrow$, a contour is of length $3k$. As there is at most $3$
choices of direction at each point of the dual graph, a contour can
be associated to a unique function from $\left\llbracket 1,3k\right\rrbracket $
to $\left\{ \rightarrow,\swarrow,\nwarrow\right\} $. Thus, $C_{k}\leq3^{3k}=27^{k}$,
and the probability that such a contour exists is less than $\sum_{k\geq1}\left(27\epsilon\right)^{k}=\frac{27\epsilon}{1-27\epsilon}$.
Thus, the probability to have an open path ending in $O$ is greater
than $1-\frac{27\epsilon}{1-27\epsilon}$.
\end{proof}

\subsubsection{Update functions}

To prove the theorem, we need an ergodicity property of this kind
of CA.
\begin{defn}
An update function $\underline{f}$ for the local rule $\varphi$
is a function such that for any $U\sim\text{Unif}([0,1])$ and $(a_{1},...,a_{r})\in\A^{r}$,
$P\left(\underline{f}(a_{1},...,a_{r},U)=b\right)=\varphi(a_{1},...,a_{r})(b)$.

A global update map $\Psi:\X\times[0,1]^{\Z}\to\X$ is defined as
$\Psi(x,u)_{k}=\underline{f}(x_{k+\mathcal{N}},u_{k})$. To simulate
the PCA, we can recursively define $\Psi^{t+1}:\X\times\left([0,1]^{\Z}\right)^{t}\to\X$
by $\Psi^{t+1}(x,u^{1},...,u^{t+1})=\Psi\left(\Psi^{t}(x,u^{1},...,u^{t}),u^{t+1}\right)$,
and give ourselves $\left(U_{i}^{n}\right)_{i,n\in\Z}$ independent
random variables uniformly distributed over $[0,1]$.
\end{defn}

\begin{prop}[\cite{MST19}, Theorem 3.11 and Proposition 3.3]
Let $F$ be a CA admitting $0$ as a spreading symbol. Then there
is an $\epsilon_{c}>0$ such that $\forall\epsilon<\epsilon_{c}$,
every $0$-positive $\epsilon$-perturbation of $F$ is uniformly
ergodic. Moreover, there is a $T\geq0$ defined uniquely by the $\left(U_{i}^{n}\right)_{i,n\in\Z}$
such that $x\mapsto\Psi^{T}(x;U^{-T},...,U^{-0})_{0}$ is almost surely
constant, with $\pi_{\epsilon}([\beta]_{0})=P\left(\Psi^{T}(\cdot;U^{-T},...,U^{-0})_{0}=\beta\right)$.
\end{prop}

Observe that in order to have $\Psi^{T}(0^{\infty};U^{-T},...,U^{-0})_{0}=0$,
it suffices to have in the spread graph defined by $\left(U^{-t}\right)_{0\leq t\leq T}$
an open path which end at the top vertex $O$ and begin at least in
the level $-T$: the symbol $0$ from this level will spread towards
the top via this open path.

\subsubsection{Proof of the theorem}
\begin{proof}[Proof of theorem \ref{thm:Spreading-general}]
Without loss of generality, we can suppose that the neighborhood
is $\mathcal{N}=\left\llbracket 0,r-1\right\rrbracket $. We denote
by $\pi_{\epsilon}$ the unique measure of $\M_{\epsilon}$. As in
the proof for the nilpotent case, the total variation distance to
a Dirac distribution is $\norm{\delta_{0}-\pi_{\epsilon}}A=1-\pi_{\epsilon}([0]_{A})$.

For any $m\in\N^{*}$, define $t_{m}=\lceil\frac{m-1}{r-1}\rceil$,
such that $A\coloneqq\llbracket a-\left|A\right|+1,a\rrbracket\subset\llbracket a-\left(r-1\right)t_{\left|A\right|},a\rrbracket\eqqcolon\tilde{A}$
and $\pi_{\epsilon}\left(\left[0\right]_{A}\right)\geq\pi_{\epsilon}\left(\left[0\right]_{\tilde{A}}\right)$.
By definition of an $\epsilon$-perturbation, we have $\pi_{\epsilon}\left(\left[0\right]_{\tilde{A}}\right)\geq\left(1-\epsilon\right)^{\left(r-1\right)t_{\left|A\right|}+1}\pi_{\epsilon}\left(\left[0\right]_{\llbracket a-\left(r-2\right)t_{\left|A\right|},a\rrbracket}\right)$.
By iterating the last inequality, we obtain

\[
\pi_{\epsilon}\left(\left[0\right]_{A}\right)\geq\pi_{\epsilon}\left(\left[0\right]_{\tilde{A}}\right)\geq\left(1-\epsilon\right)^{C_{\left|A\right|}}\pi_{\epsilon}\left(\left[0\right]_{a}\right)=\left(1-\epsilon\right)^{C_{\left|A\right|}}\pi_{\epsilon}\left(\left[0\right]_{0}\right)
\]
where $C_{|A|}=\sum_{t=1}^{t_{\left|A\right|}}\left(t(r-1)+1\right)=t_{\left|A\right|}\left(\frac{r-1}{2}\left(t_{\left|A\right|}+1\right)+1\right)$,
which corresponds to the number of cells where it suffices to not
have any mistake to be sure to have only the symbol $0$ in all the
cells of $A$ at $t=0$ (see Figure \ref{fig:Spreading}) . For $\left|A\right|\gg1$,
we have $C_{\left|A\right|}\sim\frac{\left|A\right|^{2}}{2\left(n-1\right)}$.
\begin{figure}
\centering{}\def\svgwidth{\columnwidth * 4/5}
\begingroup%
  \makeatletter%
  \providecommand\color[2][]{%
    \errmessage{(Inkscape) Color is used for the text in Inkscape, but the package 'color.sty' is not loaded}%
    \renewcommand\color[2][]{}%
  }%
  \providecommand\transparent[1]{%
    \errmessage{(Inkscape) Transparency is used (non-zero) for the text in Inkscape, but the package 'transparent.sty' is not loaded}%
    \renewcommand\transparent[1]{}%
  }%
  \providecommand\rotatebox[2]{#2}%
  \newcommand*\fsize{\dimexpr\f@size pt\relax}%
  \newcommand*\lineheight[1]{\fontsize{\fsize}{#1\fsize}\selectfont}%
  \ifx\svgwidth\undefined%
    \setlength{\unitlength}{603.59682579bp}%
    \ifx\svgscale\undefined%
      \relax%
    \else%
      \setlength{\unitlength}{\unitlength * \real{\svgscale}}%
    \fi%
  \else%
    \setlength{\unitlength}{\svgwidth}%
  \fi%
  \global\let\svgwidth\undefined%
  \global\let\svgscale\undefined%
  \makeatother%
  \begin{picture}(1,0.48625426)%
    \lineheight{1}%
    \setlength\tabcolsep{0pt}%
    \put(0,0){\includegraphics[width=\unitlength,page=1]{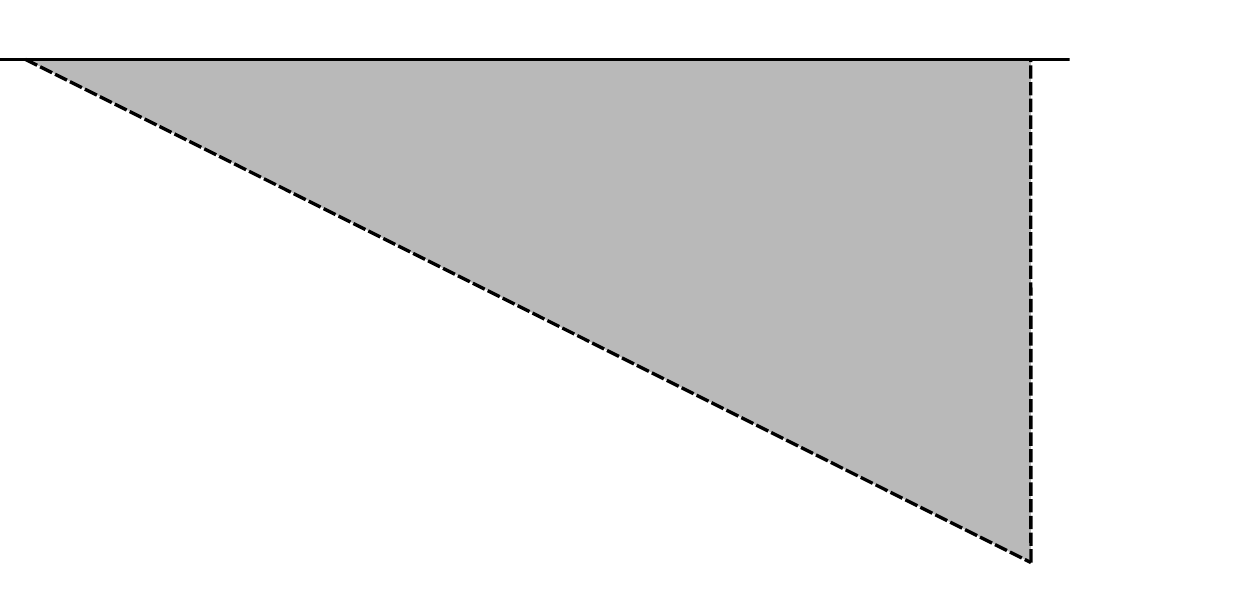}}%
    \put(0.88997162,0.43249253){\color[rgb]{0,0,0}\makebox(0,0)[t]{\lineheight{1.25}\smash{\begin{tabular}[t]{c}$t=0$\\\end{tabular}}}}%
    \put(0.81904042,0.00482582){\color[rgb]{0,0,0}\makebox(0,0)[t]{\lineheight{1.25}\smash{\begin{tabular}[t]{c}$a$\end{tabular}}}}%
    \put(0,0){\includegraphics[width=\unitlength,page=2]{spreadingConstante.pdf}}%
    \put(0.91904841,0.03245395){\color[rgb]{0,0,0}\makebox(0,0)[t]{\lineheight{1.25}\smash{\begin{tabular}[t]{c}$t=-t_{|A|}$\\\end{tabular}}}}%
    \put(0,0){\includegraphics[width=\unitlength,page=3]{spreadingConstante.pdf}}%
    \put(0.46259685,0.39937108){\color[rgb]{0,0.44313725,0}\makebox(0,0)[t]{\lineheight{1.25}\smash{\begin{tabular}[t]{c}$A$\end{tabular}}}}%
    \put(0,0){\includegraphics[width=\unitlength,page=4]{spreadingConstante.pdf}}%
    \put(0.42532027,0.46093116){\color[rgb]{0.44313725,0.00392157,0}\makebox(0,0)[t]{\lineheight{1.25}\smash{\begin{tabular}[t]{c}$\tilde{A}=\llbracket a-\left(r-1\right)t_{\left|A\right|},a\rrbracket$\end{tabular}}}}%
  \end{picture}%
\endgroup%
\caption{\label{fig:Spreading}Proof of theorem \ref{thm:Spreading-general}.
To have a $0$ when $t=0$ on all $A$, it suffices to have one in
$\left(-t_{\left|A\right|},a\right)$ and not make any mistakes on
the cells of the colored area. Time goes upward.}
\end{figure}

By the previous propositions $x\mapsto\Psi^{T}(x;...)$ is constant.
One can then only use its value on the entry $x=0^{\infty}$. 
\begin{align*}
\pi_{\epsilon}([0]_{0}) & =P\left(\Psi^{T}(\cdot;U^{-T},...,U^{-0})_{0}=0\right)\\
 & =P\left(\Psi^{T}(0^{\infty};U^{-T},...,U^{-0})_{0}=0\right)\\
 & \geq P\left(\text{There is an open path from level }-T\text{ to \ensuremath{O} in the graph defined by }\left(U^{-t}\right)_{t\geq0}\right)\\
 & \geq P\left(\text{There is an infinite open path from level ending in }O\text{ in the graph defined by }\left(U^{-t}\right)_{t\geq0}\right)\\
 & \geq1-\frac{27\epsilon}{1-27\epsilon}.
\end{align*}
where the final inequality is by percolation.
\end{proof}
\begin{example}
For the simple case $\mathcal{N}=\left\{ 0,1\right\} $, the result
is 
\[
\norm{\delta_{0}-\pi_{\epsilon}}A\leq1-\left(1-\epsilon\right)^{\frac{\left(\left|A\right|+2\right)(\left|A\right|-1)}{2}}\cdot\left(1-\frac{27\epsilon}{1-27\epsilon}\right)\underset{\epsilon\to0}{\sim}\left(\frac{\left(\left|A\right|+2\right)(\left|A\right|-1)}{2}+27\right)\epsilon.
\]
The spread graph used in the proof is much simpler (Figure \ref{fig:Percolation01}),
and is the one described in \cite{Toom01}.
\end{example}

\begin{figure}[H]
\resizebox{.5\linewidth}{!}{     \begin{tikzpicture}[ err/.style={circle, draw=red!60, very thick, minimum size=3mm}, cell/.style={circle, draw=blue!60, very thick, minimum size=3mm}, noir/.style={circle, draw=black!60, very thick, minimum size=3mm}, ]         \node[cell] at (1,0) (0 0) {};         \node[cell] at (3,0) (0 1) {};         \node[cell] at (5,0) (0 2) {};                  \node[err] at (2,.5) (1 0p) {};         \node[err] at (4,.5) (1 1p) {};                  \node[cell] at (2,1.5) (1 0) {};         \node[cell] at (4,1.5) (1 1) {};                  \node[err] at (3,2) (2 0p) {};         \node[cell] at (3,3.1) (2 0) {$O$};                  \draw[->, color = blue] (0 0) -- (1 0p);         \draw[->, color = blue] (0 1) -- (1 0p);         \draw[->, color = blue] (0 1) -- (1 1p);         \draw[->, color = blue] (0 2) -- (1 1p);         \draw[->, color = red] (1 0p) -- (1 0);         \draw[->, color = red] (1 1p) -- (1 1);         \draw[->, color = blue] (1 0) -- (2 0p);         \draw[->, color = blue] (1 1) -- (2 0p);         \draw[->, color = red] (2 0p) -- (2 0);
        \node[noir] at (0,-0.5) (d0 0) {};         \node[noir] at (2,-0.5) (d0 1) {};         \node[noir] at (4,-0.5) (d0 2) {};         \node[noir] at (6,-0.5) (d0 3) {};                  \node[noir] at (1,1) (d1 0) {};         \node[noir] at (3,1) (d1 1) {};         \node[noir] at (5,1) (d1 2) {};                  \node[noir] at (2,2.5) (d2 0) {};         \node[noir] at (4,2.5) (d2 1) {};                  \node[noir] at (3,4) (d3 0) {};         
        \draw[->, dashed, color = red] (d0 0) -- (d0 1);         \draw[->, dashed, color = red] (d0 1) -- (d0 2);         \draw[->, dashed, color = red] (d0 2) -- (d0 3);                  \draw[->, dashed, color = red] (d1 0) -- (d1 1);         \draw[->, dashed, color = red] (d1 1) -- (d1 2);                  \draw[->, dashed, color = red] (d2 0) -- (d2 1);
        \draw[->] (d3 0) -- (d2 0);         \draw[->] (d2 0) -- (d1 0);         \draw[->] (d1 0) -- (d0 0);                  \draw[->] (d2 1) -- (d1 1);         \draw[->] (d1 1) -- (d0 1);                  \draw[->] (d1 2) -- (d0 2);                  \draw[->] (d0 3) -- (d1 2);         \draw[->] (d1 2) -- (d2 1);         \draw[->] (d2 1) -- (d3 0);                  \draw[->] (d0 2) -- (d1 1);         \draw[->] (d1 1) -- (d2 0);                  \draw[->] (d0 1) -- (d1 0);                       \end{tikzpicture}     }
\centering{}\caption{\label{fig:Percolation01}Illustration for neighborhood $\mathcal{N}=\{0,1\}$.
In blue the spread graph, in black the dual graph. Vertical red arrows
have a probability $\epsilon$ to be \textquotedblleft closed\textquotedblright ,
and thus the horizontal dashed red ones have a probability $\epsilon$
to be \textquotedblleft open\textquotedblright .}
\end{figure}
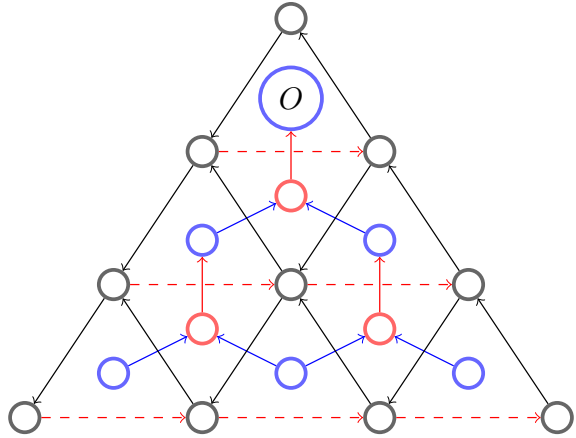

\subsection{$d$-dimensional binary \label{subsec:d-dimensional-binary}}

For $\A=\{0,1\}$, we define $F$ such that $(Fx)_{0}=\begin{cases}
0 & \text{if }x_{i}=0\text{ for any \ensuremath{i\in\mathcal{N}}}\\
1 & \text{otherwise}
\end{cases}$: the symbol $0$ is spreading. We will prove the stochastic stability
of $\delta_{0}$ using Fourier analysis. In \cite{MST19}, the authors
prove that a perturbation of $F$ by a zero-range noise is (under
certain circumstances) ergodic.
\begin{thm}
\label{thm:BinaireSpreading}Let $F$ defined as above, and $F_{\epsilon}$
a zero-range perturbation of $F$, with noise matrix $\theta_{\epsilon}=\begin{pmatrix}1-p_{\epsilon} & p_{\epsilon}\\
q_{\epsilon} & 1-q_{\epsilon}
\end{pmatrix}$ such that $p_{\epsilon}\leq\epsilon$, and $0<q_{\epsilon}\leq\epsilon$.
Let $\nu_{\epsilon}$ be the unique $F_{\epsilon}$-invariant measure
(for $\epsilon<\frac{1}{2}$). Then for all finite $A\subset\Z^{d}$,
\[
\norm{\delta_{0}-\nu_{\epsilon}}A\leq\left(2^{\left|A\right|}-1\right)\frac{p_{\epsilon}^{\left|\mathcal{N}\right|}}{q_{\epsilon}}.
\]
 In particular, $\delta_{0}$ is stochastically stable if $\frac{p_{\epsilon}^{\left|\mathcal{N}\right|}}{q_{\epsilon}}\tendsto{\epsilon}00$.
\end{thm}

\begin{rem}
The condition $q_{\epsilon}>0$ implies a $0$-positive perturbation,
so we already know that $\delta_{0}$ is stochastically stable if
we are in the case $d=1$, with a speed of convergence that is at
least linear. Depending on the value of $\left|\mathcal{N}\right|$
and the ratio $\frac{p_{\epsilon}^{\left|\mathcal{N}\right|}}{q_{\epsilon}}$,
the conclusion may be stronger or weaker than the previous theorem.
In fact, in the case $\frac{p_{\epsilon}^{\left|\mathcal{N}\right|}}{q_{\epsilon}}\underset{\epsilon\to0}{\not\to}0$,
this theorem tells nothing on the stochastic stability of $\delta_{0}$:
as mentioned in the beginning of the section, we know that $\delta_{0}$
is stochastically stable as a direct consequence of the stable eroder
nature of the CA (as proven in \cite{Toom80}) and the uniform ergodicity
of its perturbation (as proven in \cite{MST19}).
\end{rem}

To prove this theorem, we use the Möbius basis of $C_{0}(\X)$.
\begin{defn}
Let $\chi:\A\to\C$ be defined by $\chi(0)=0$, $\chi(1)=1$. For
a finite $A\subset\Z^{d}$, let $\chi_{A}:\X\to\C$ be defined by
\[
\chi_{A}:x\mapsto\prod_{i\in A}\chi(x_{i})=\begin{cases}
1 & \text{if }x_{i}=1\,\forall i\in A\\
0 & \text{otherwise}
\end{cases}.
\]
The set of all $\chi_{A}$ for finite $A$ forms a basis of $C_{0}(\X)$.
For any observable $h\in C_{0}(\X)$, we note its decomposition on
this basis as $h=\sum_{A\subset\Z^{d}}\hat{h}_{A}\chi_{A}$. We finally
define a semi-norm associated to it, 
\[
\left\langle \left\langle h\right\rangle \right\rangle \coloneqq\sum_{\emptyset\neq A\subset\Z^{d}}\left|\hat{h}_{A}\right|.
\]
\end{defn}

For a given $h\in C_{0}(\X)$, we have $\hat{h}_{\emptyset}=\int h(x)\diff\delta_{0}(x)=h\left(0^{\infty}\right)$.
\begin{prop}[\cite{MST19} Theorem 5.3]
\label{prop:PhiChiBinarySpreading}For any finite $A\subset\Z^{d}$
, we have $F_{\epsilon}\chi_{A}=\sum_{I\subset A}p_{\epsilon}^{\left|A\backslash I\right|}(1-p_{\epsilon}-q_{\epsilon})^{\left|I\right|}\chi_{I}$.
\end{prop}

With these tools in hand, can finally prove our theorem.
\begin{proof}[Proof of theorem \ref{thm:BinaireSpreading}]
As $q_{\epsilon}>0$ and $p_{\epsilon}+q_{\epsilon}\leq2\epsilon<1$,
the measure $\nu_{\epsilon}$ is well-defined and $F_{\epsilon}$
is uniformly ergodic: in particular for all probability measure $\mu$
on $\X$, $F_{\epsilon}^{t}\mu\weakto t{\infty}\nu_{\epsilon}$. By
linearity, $F_{\epsilon}h=\sum_{A}\hat{h}_{A}\left(F_{\epsilon}\chi_{A}\right)$
and by proposition \ref{prop:PhiChiBinarySpreading}, one deduces
that $\left\langle \left\langle F_{\epsilon}\chi_{A}\right\rangle \right\rangle \leq(1-q_{\epsilon})^{\left|A+\mathcal{N}\right|}\leq1-q_{\epsilon}$
for $A\neq\emptyset$, and $\left\langle \left\langle F_{\epsilon}h\right\rangle \right\rangle \leq\left(1-q_{\epsilon}\right)\left\langle \left\langle h\right\rangle \right\rangle $.
\\
Also, $\left(\hat{F_{\epsilon}h}\right)_{\emptyset}=\hat{h}_{\emptyset}+\sum_{\emptyset\neq A}\hat{h}_{A}p_{\epsilon}^{\left|A+\mathcal{N}\right|}$.
Then, with $h^{t}=F_{\epsilon}^{t}h$, one gets 
\[
\left|\hat{h^{t+1}}_{\emptyset}-\hat{h^{t}}_{\emptyset}\right|\leq p_{\epsilon}^{\left|\mathcal{N}\right|}\left\langle \left\langle h^{t}\right\rangle \right\rangle \leq p_{\epsilon}^{\left|\mathcal{N}\right|}\left(1-q_{\epsilon}\right)^{t}\left\langle \left\langle h\right\rangle \right\rangle .
\]
And so 
\[
\left|\hat{h^{t}}_{\emptyset}-\hat{h}_{\emptyset}\right|\leq p_{\epsilon}^{\left|\mathcal{N}\right|}\sum_{i=0}^{t-1}\left(1-q_{\epsilon}\right)^{i}\left\langle \left\langle h\right\rangle \right\rangle \leq\frac{p_{\epsilon}^{\left|\mathcal{N}\right|}}{q_{\epsilon}}\left\langle \left\langle h\right\rangle \right\rangle .
\]
Taking the limit, $\hat{h^{t}}_{\emptyset}=\int h(x)\diff\left(F^{t}\delta_{0}\right)(x)\tendsto t{\infty}\int h(x)\diff\nu_{\epsilon}(x)$
and thus $\left|\int h\diff\nu_{\epsilon}-\int h\diff\delta_{0}\right|\leq\frac{p_{\epsilon}^{\left|\mathcal{N}\right|}}{q_{\epsilon}}\left\langle \left\langle h\right\rangle \right\rangle $.

In the case $h=\mathbf{1}_{[0]_{A}}=\sum_{I\subset A}(-1)^{\left|I\right|}\chi_{I}$,
$\left\langle \left\langle h\right\rangle \right\rangle =2^{A}-1$.
Therefore, 
\begin{align*}
\norm{\delta_{0}-\nu_{\epsilon}}A & =1-\nu_{\epsilon}([0]_{A})\\
 & =\int\mathbf{1}_{[0]_{A}}\diff\delta_{0}-\int\mathbf{1}_{[0]_{A}}\diff\nu_{\epsilon}\\
 & \leq\left(2^{\left|A\right|}-1\right)\frac{p_{\epsilon}^{\left|\mathcal{N}\right|}}{q_{\epsilon}}.
\end{align*}
\end{proof}
\begin{example}
For a uniform noise, i.e. where $p_{\epsilon}=q_{\epsilon}=\epsilon$,
one gets a speed of convergence in $\epsilon^{\left|\mathcal{N}\right|-1}$,
therefore a linear speed again for $\mathcal{N}=\{0,1\}$. 
\end{example}

\section{Computational lemmas \label{sec:lemmas}}

This section is dedicated to the proof of the two following results,
Proposition \ref{prop:equivalentaubord} and \ref{prop:limiteaubord}.
We will use them in the last section of the article to prove the undecidability
of stochastic stability. The proofs are purely computational, and
don't give much insight on the main theorem. 
\begin{prop}
\label{prop:equivalentaubord}For all $\alpha>-1$ and $\beta\in\N^{*}$,
the following holds: 
\[
\sum_{n\geq0}n^{\alpha}x^{n^{\beta}}\underset{x\to1^{-}}{\sim}\frac{\Gamma\left(\frac{1+\alpha}{\beta}\right)}{\beta\left(1-x\right)^{\frac{1+\alpha}{\beta}}}
\]
 where $\Gamma$ is the gamma function defined by $\Gamma\left(z\right)=\int_{0}^{+\infty}t^{z-1}\e^{t}\diff t$.
\end{prop}

\begin{prop}
\label{prop:limiteaubord}Let $C>0$ and $a\in\N,c\in\N^{*}$ be such
that $a\leq2c$. Then, 
\[
\sum_{n\geq0}\left(1-\left(1-\frac{\epsilon}{C}\right)^{an}\right)\left(1-\epsilon\right)^{cn^{2}}\tendsto{\epsilon}0\frac{1}{2C}\cdot\frac{a}{c}.
\]
\end{prop}

Both results are consequence of the following classical lemma and
its corollary, which proofs are in Appendix \ref{sec:Demonstrations-lemmas}.
\begin{lem}
\label{lem:abelien}Let $\left(a_{n}\right)_{n\in\N}$ and $\left(b_{n}\right)_{n\in\N}$
be such that $\sum a_{n}x^{n}$ and $\sum b_{n}x^{n}$ are power series
with convergence radius greater or equal to $1$, $b_{n}>0$ and $\sum b_{n}$
diverges. If $\frac{a_{n}}{b_{n}}\tendsto n{\infty}l\in\C$, then
\[
\frac{\sum^{\infty}a_{n}x^{n}}{\sum^{\infty}b_{n}x^{n}}\tendsto x{1^{-}}l.
\]
\end{lem}

\begin{cor}
\label{cor:abelien}Define $A_{n}=\sum_{k=0}^{n}a_{k}$ and $B_{n}=\sum_{k=0}^{n}b_{k}$.
If $\frac{A_{n}}{B_{n}}\tendsto n{\infty}l\in\C$, then
\[
\frac{\sum^{\infty}a_{n}x^{n}}{\sum^{\infty}b_{n}x^{n}}\tendsto x{1^{-}}l.
\]
\end{cor}

By direct induction, one can generalize to the case $\frac{\sum^{n}A_{k}}{\sum^{n}B_{k}}\tendsto n{\infty}l$,
etc.
\begin{proof}[Proof of Proposition \ref{prop:equivalentaubord}]
Consider first the case $\beta=1$. Standard calculations (see for
example \cite[chapter VI.2]{FS09}) gives $\frac{1}{\left(1-x\right)^{1+\alpha}}=\sum_{n=0}^{+\infty}\frac{\Gamma\left(n+1+\alpha\right)}{\Gamma\left(1+\alpha\right)\Gamma\left(n+1\right)}x^{n}.$
The Stirling formula $\Gamma\left(x+1\right)\underset{x\to\infty}{\sim}\sqrt{2\pi x}\left(\frac{x}{e}\right)^{x}$
leads us to define $b_{n}\coloneqq\frac{n^{\alpha}}{\Gamma\left(1+\alpha\right)}$,
which verifies $a_{n}\coloneqq\frac{\Gamma\left(n+1+\alpha\right)}{\Gamma\left(1+\alpha\right)\Gamma\left(n+1\right)}\underset{n\to\infty}{\sim}b_{n}$
and the hypotheses of Lemma \ref{lem:abelien}. Thus, $\sum_{n\geq0}n^{\alpha}x^{n}\underset{x\to1^{-}}{\sim}\frac{\Gamma\left(1+\alpha\right)}{\left(1-x\right)^{1+\alpha}}$.

For the general case $\beta\in\N^{*}$, define $\gamma=\frac{1+\alpha}{\beta}-1>-1$.
Using the previous case, $\frac{\Gamma\left(\frac{1+\alpha}{\beta}\right)}{\beta\left(1-x\right)^{\frac{1+\alpha}{\beta}}}=\frac{\Gamma\left(1+\gamma\right)}{\beta\left(1-x\right)^{1+\gamma}}\underset{x\to1^{-}}{\sim}\sum_{n\geq0}\frac{n^{\gamma}}{\beta}x^{n}$.
Define $b_{n}=\frac{n^{\gamma}}{\beta}$ and $a_{n}=\begin{cases}
k^{\alpha} & \text{if }n=k^{\beta}\\
0 & \text{otherwise}
\end{cases}$, and their respective cumulative sums $B_{n}$ and $A_{n}$. By integral-sums
comparison, one has
\begin{eqnarray*}
B_{n}=\sum_{k=0}^{n}\frac{k^{\gamma}}{\beta}\sim\frac{n^{\gamma+1}}{\beta\left(\gamma+1\right)}=\frac{n^{\frac{1+\alpha}{\beta}}}{1+\alpha} & \text{ and } & A_{n}=\sum_{k=0}^{\left\lfloor n^{1/\beta}\right\rfloor }k^{\alpha}\sim\frac{\left(n^{1/\beta}\right)^{1+\alpha}}{1+\alpha}=\frac{n^{\frac{1+\alpha}{\beta}}}{1+\alpha}.
\end{eqnarray*}
The previous corollary gives the wanted result. 
\end{proof}
For the second proposition, the idea is to split the sum in two, and
use the lemma to produce a asymptotic development for each. 
\begin{lem}
One has the following asymptotic development: 
\[
\sum_{n\geq0}x^{n^{2}}\underset{x\to1^{-}}{=}\frac{1}{2}\sqrt{\frac{\pi}{1-x}}+\frac{1}{2}+o\left(1\right).
\]
\end{lem}

\begin{proof}
Denote by $a_{i}=\begin{cases}
1 & \text{if }i\text{ is a square}\\
0 & \text{otherwise}
\end{cases}$ and by $b_{i}=\begin{cases}
\frac{\sqrt{\pi}}{2} & \text{if }n=0\\
\frac{\sqrt{\pi}}{2}\cdot\frac{\left(2n-1\right)!}{\left(2^{n}n!\right)^{2}} & \text{otherwise}
\end{cases}.$ Their respective cumulative sums are denoted by $A_{k}=\sum_{i=0}^{k}a_{i}$
and $B_{k}=\sum_{i=0}^{k}b_{i}$. Consider then the following radius-1
power series: $\sum_{n=0}^{\infty}x^{n^{2}}=\sum_{i=0}^{\infty}a_{i}x^{i}\eqqcolon f(x)$
and $\frac{1}{2}\sqrt{\frac{\pi}{1-x}}=\sum_{i=0}^{\infty}b_{i}x^{i}\eqqcolon g(x)$.
Finally, define $S_{n}=\sum_{k=0}^{n}A_{k}-B_{k}$. 

By Lemma \ref{lem:abelien}, it suffices to prove that $\frac{S_{n}}{n}\tendsto n{\infty}\frac{1}{2}$.
Indeed, it follows that $\frac{f(x)-g(x)}{(1-x)^{2}}=\sum S_{n}x^{n}\underset{x\to1^{-}}{\sim}\sum\frac{n}{2}x^{n}=\frac{x}{2(1-x)^{2}}$,
and thus $f(x)-g(x)\tendsto x{1^{-}}\frac{1}{2}$.
\begin{enumerate}
\item Computations shows that $A_{n}=1+\left\lfloor \sqrt{n}\right\rfloor $
and $B_{n}=\frac{\sqrt{\pi}}{2}\cdot\frac{\left(2n+1\right)!}{\left(2^{n}n!\right)^{2}}$. 
\item Because for $m\in\left\llbracket n^{2},\left(n+1\right)^{2}\right\llbracket ,\,A_{m}=1+n$,
we have $\sum_{k=0}^{2n}A_{n^{2}+k}=\left(2n+1\right)(n+1)=2\left(n+1\right)^{2}-(n+1)$
and
\begin{align*}
\sum_{m=0}^{N^{2}-1}A_{m} & =\sum_{n=0}^{N-1}\sum_{k=0}^{2n}A_{n^{2}+k}\\
 & =2\sum_{n=0}^{N-1}\left(n+1\right)^{2}-\sum_{n=0}^{N-1}\left(n+1\right)\\
 & =\frac{N(N+1)(2N+1)}{3}-\frac{N(N+1)}{2}\\
\sum_{m=0}^{N^{2}-1}A_{m} & =\frac{2}{3}N^{3}+\frac{1}{2}N^{2}+O(N).
\end{align*}
The result is the same if adding until $N^{2}$ as $A_{N^{2}}=O(N)$.
\item Using Stirling formula, $B_{n}=\frac{\sqrt{\pi}}{2}\cdot\frac{\left(2n+1\right)!}{\left(2^{n}n!\right)^{2}}=\sqrt{n}\left(1+\frac{3}{8n}+O\left(\frac{1}{n^{2}}\right)\right)$.
\\
Series-integral comparison gives $u_{n}=\sum_{k=0}^{n}\sqrt{k}\underset{n\to\infty}{\sim}\frac{2}{3}n^{3/2}$.
If we define $t_{n}$ as $t_{n}=u_{n}-$$\frac{2}{3}n^{3/2}$, one
has $t_{n}-t_{n-1}\sim\frac{1}{4\sqrt{n}}$. By adding the comparison
relations, it follows that $t_{n}\sim\sum_{k=1}^{n}\frac{1}{4\sqrt{k}}\sim\frac{\sqrt{n}}{2}$.
Thus $u_{n}=\frac{2}{3}n^{3/2}+O\left(\sqrt{n}\right)$.\\
Moreover, 
\begin{align*}
\sum_{k=0}^{N^{2}}B_{k}-\sqrt{k} & \underset{N\to\infty}{\sim}\frac{3}{8}\sum_{k=1}^{N^{2}}\frac{1}{\sqrt{k}} & \text{Addding the comparison relations}\\
 & \underset{N\to\infty}{\sim}\frac{3}{8}\cdot\frac{\sqrt{N^{2}}}{2} & \sum-\int\text{ comparison}\\
\sum_{k=0}^{N^{2}}B_{k}-\sqrt{k} & \underset{N\to\infty}{=}O(N)
\end{align*}
Combining the two results yields 
\[
\sum_{k=0}^{N^{2}}B_{k}\underset{N\to\infty}{=}\frac{2}{3}N^{3}+O(N).
\]
\item Thus, 
\[
\frac{S_{n^{2}}}{n^{2}}=\frac{\sum_{k=0}^{N^{2}}A_{k}-B_{k}}{N^{2}}\tendsto N{\infty}\frac{1}{2}.
\]
\item Decompose every $m\in\N$ as $m=n^{2}+k$ with $n=\left\lfloor \sqrt{m}\right\rfloor $
and $k\in\left\llbracket 0,2n\right\rrbracket $. One gets 
\[
\frac{S_{m}}{m}=\frac{S_{n^{2}}}{n^{2}+k}+\frac{\sum_{i=1}^{k}A_{n^{2}+i}-B_{n^{2}+i}}{n^{2}+k}.
\]
In one hand, $\frac{S_{n^{2}}}{n^{2}+k}\underset{m\to\infty}{\sim}\frac{S_{n^{2}}}{n^{2}}\underset{m\to\infty}{\sim}\frac{1}{2}$.
In the other hand, as $\left(A_{k}\right)$ and $\left(B_{k}\right)$
are non-decreasing one gets 
\[
\underset{=O(n)}{\underbrace{k}}\cdot\underset{=O(1)}{\underbrace{\left(n+1-B_{\left(n+1\right)^{2}}\right)}}\leq\sum_{i=1}^{k}A_{n^{2}+i}-B_{n^{2}+i}\leq\underset{=O(n)}{\underbrace{k}}\cdot\underset{=O(1)}{\underbrace{\left(n+1-B_{n^{2}}\right)}}
\]
So $\frac{\sum_{i=1}^{k}A_{n^{2}+i}-B_{n^{2}+i}}{n^{2}+k}\tendsto m00$
and finally $\frac{S_{m}}{m}\tendsto m{\infty}\frac{1}{2}$. 
\end{enumerate}
\end{proof}
\begin{cor}
\label{cor:dvpmtAsym1}For $c\in\N^{*}$, one has the following asymptotic
development: 
\[
\sum_{n\geq0}x^{cn^{2}}\underset{x\to1^{-}}{=}\frac{1}{2}\sqrt{\frac{\pi}{c\left(1-x\right)}}+\frac{1}{2}+o\left(1\right).
\]
\end{cor}

\begin{proof}
By \textbf{$\left(1-x^{c}\right)\underset{x\to1}{\sim}c\left(1-x\right)$.}
\end{proof}
Define $\omega_{m}$ for fixed $a\leq2c$, $m\in\N$, as $\omega_{m}=\begin{cases}
\binom{an}{k}\frac{\left(C-1\right)^{an-k}}{C^{an}} & \text{if }m=cn^{2}+k\text{ with }0\leq k\leq an\\
0 & \text{otherwise}
\end{cases}$. One can verify that the condition $a\leq2c$ is enough to have $\omega_{m}$
well-defined.
\begin{lem}
\label{lem:dvpmtAsym2}For $C>0$, $a\in\N,c\in\N^{*}$ such that
$a\leq2c$, one has the following asymptotic development: 
\[
\sum_{m=0}^{+\infty}\omega_{m}x^{m}\underset{x\to1^{-}}{=}\frac{1}{2}\sqrt{\frac{\pi}{c\left(1-x\right)}}+\frac{1}{2}-\frac{a}{2Cc}+o\left(1\right).
\]
\end{lem}

\begin{proof}
\textcolor{red}{ }

The proof is similar to the previous one, with $a_{i}=\omega_{i}$
and $b_{i}=\begin{cases}
\frac{\sqrt{\pi}}{2\sqrt{c}} & \text{if }n=0\\
\frac{\sqrt{\pi}}{2\sqrt{c}}\cdot\frac{\left(2n-1\right)!}{\left(2^{n}n!\right)^{2}} & \text{otherwise}
\end{cases}$. Define similarly $A_{k},B_{k}$ and $S_{n}$.
\begin{enumerate}
\item Computations leads to $B_{n}=\frac{\sqrt{\pi}}{2\sqrt{c}}\cdot\frac{\left(2n+1\right)!}{\left(2^{n}n!\right)^{2}}$.
By $\sum_{k=0}^{an}\omega_{cn^{2}+k}=1$, one obtains $A_{cn{{}^2}-1}=n$.
\item Computations leads to
\begin{align*}
\sum_{m=0}^{cN^{2}-1}A_{m} & =c\frac{N(N-1)(2N-1)}{3}+\left(c+2c-\frac{a}{C}\right)\frac{N(N-1)}{2}+Nc\\
 & =\frac{2}{3}cN^{3}+\left(\frac{1}{2}-\frac{a}{2Cc}\right)cN^{2}+O(N).
\end{align*}
\item Similarly, 
\begin{align*}
\sum_{k=0}^{cN^{2}}B_{k} & \underset{N\to\infty}{=}\frac{2}{3}\cdot\frac{1}{\sqrt{c}}\left(\sqrt{c}N\right)^{3}+O\left(N\right)\\
 & \underset{N\to\infty}{=}\frac{2}{3}cN^{3}+O\left(N\right).
\end{align*}
\item Thus, 
\[
\frac{S_{cN^{2}}}{cN^{2}}=\frac{\sum_{k=0}^{N^{2}}A_{k}-B_{k}}{cN^{2}}\tendsto N{\infty}\frac{1}{2}-\frac{a}{2Cc}.
\]
\item With the decomposition $m=cn^{2}+k$ where $k\in\llbracket0,2n\rrbracket$,
it suffices to observe that
\[
\underset{=O(n)}{\underbrace{\left(k+1\right)}}\cdot\underset{=O(1)}{\underbrace{\left(n-B_{c\left(n+1\right)^{2}}\right)}}\leq\sum_{i=0}^{k}A_{cn^{2}+i}-B_{cn^{2}+i}\leq\underset{=O(n)}{\underbrace{\left(k+1\right)}}\cdot\underset{=O(1)}{\underbrace{\left(n+1-B_{cn^{2}}\right)}}
\]
to conclude in the same vein that $\frac{S_{m}}{m}\tendsto m{\infty}\frac{1}{2}-\frac{1}{C}$.
\end{enumerate}
\end{proof}
\begin{proof}[Proof of Proposition \ref{prop:limiteaubord}]
 Define $x=1-\epsilon$, so that $\left(1-\frac{\epsilon}{C}\right)=\frac{1}{C}\left(C-1+x\right)$.
Thus,
\[
\left(1-\frac{\epsilon}{C}\right)^{an}=\frac{1}{C^{an}}\sum_{k=0}^{an}\binom{an}{k}x^{k}\left(C-1\right)^{an-k}.
\]

Decompose $\sum_{n\geq0}\left(1-\left(1-\frac{\epsilon}{C}\right)^{an}\right)\left(1-\epsilon\right)^{cn^{2}}=\sum_{n\geq0}\left(1-\epsilon\right)^{cn^{2}}-\sum_{n\geq0}\left(1-\frac{\epsilon}{C}\right)^{an}\left(1-\epsilon\right)^{cn^{2}}$.
The second sum can be rewritten as

\begin{align*}
\sum_{n\geq0}\left(1-\frac{\epsilon}{C}\right)^{an}\left(1-\epsilon\right)^{cn^{2}} & =\sum_{n\geq0}^{+\infty}\sum_{k=0}^{an}\binom{an}{k}\frac{\left(C-1\right)^{an-k}}{C^{an}}x^{cn^{2}+k}\\
 & =\sum_{m=0}^{+\infty}\omega_{m}x^{m}
\end{align*}
where $\omega_{m}=\begin{cases}
\binom{an}{k}\frac{\left(C-1\right)^{an-k}}{C^{an}} & \text{if }m=cn^{2}+k\text{ with }0\leq k\leq an\\
0 & \text{otherwise}
\end{cases}$. By Corollary \ref{cor:dvpmtAsym1} and Lemma \ref{lem:dvpmtAsym2},
one can obtain their respective asymptotic development. Using them
both leads to

\[
\sum_{n\geq0}\left(1-\left(1-\frac{\epsilon}{C}\right)^{an}\right)\left(1-\epsilon\right)^{cn^{2}}\tendsto{\epsilon}0\frac{a}{2Cc}.
\]
\end{proof}

\section{Undecidability}

The purpose of this section is to show that given a cellular automaton, it is undecidable to know if it is strongly stochastically stable. Thus we cannot hope to have a simple characterization of this phenomenon. 

\begin{thm}
\label{thm:undecidabilite}For $\epsilon>0$ and $F$ a CA, denote
by $F_{\epsilon}$ its perturbation by a uniform noise of scale $\epsilon$.

The problem which take in input the rule of a cellular automaton $F$ and say that $\delta_{0}$ is stochastically stable under $\left\{ F_{\epsilon}\right\} _{\epsilon>0}$ is undecidable.
\end{thm}

To prove the theorem, we simulate a Turing machine in the CA such
that the $0$ wins if and only if the machine halts. The construction
is heavily inspired by the one described in \cite[section 3]{BDPST15}
and \cite[section 5]{ENT22}. In short, a special symbol $*$ is used
to initialize the machine, and create a cone where the calculations
occur, protected from outside $0$s. If the machine halts, it create
a $0$ inside the cone that quickly erase the latter.

The main idea behind the construction is that if the machine halts,
the cones disappear in a finite amount of time so the $0$ ``should
win''. If the machine does not halt, the cones are infinite and stops
the $0$. The errors are both useful and a problem: they are the one
that make the $*$ symbols appear in the first place, but can perturb
the computation of the machine.

In this section, we first recall some basic notions about Turing machines,
and then describe with more details the CA and its perturbation. The
last parts deals with the two cases, when the Turing machine halts
in a finite amount of time or not.

\subsection{Definition of a Turing machine}

For a recent broader study on the subject of Turing machines and its
applications, see for example \cite{MP22,Odifreddi92}.\textcolor{brown}{{}
}A Turing machine is one of many computation model. Consider a bi-infinite
tape (indexed by $\Z$) where on each cell is inscribed a symbol $\gamma\in\Gamma=\mathcal{B}\sqcup\left\{ \emptyset\right\} $,
where $\mathcal{B}$ is a finite alphabet. Denote by $Q$ a finite
set of state of the head of the machine, containing $q_{\bot}$ a
halting state and $q_{0}$ an initial state. Finally, define by $\delta:Q\times\Gamma\to Q\times\Gamma\times\left\{ \leftarrow,\rightarrow\right\} $
the transition function of the machine. The Turing machine in itself
is the tuple $(Q,q_{0},q_{\bot},\Gamma,\delta)$.

Initially, the head is positioned at the cell indexed by $0$ in the
state $q_{0}$, while each cell of the tape is inscribed with the
empty symbol $\emptyset$. At each step, the head with state $q$
read the symbol $\gamma$ on the cell it is on, then follows the instruction
of the transition function $\delta(q,\gamma)=\left(q^{\prime},\gamma^{\prime},d\right)$:
the head takes the state $q^{\prime}$, replace $\gamma$ by $\gamma^{\prime}$
on the cell it is on, and take a step in the direction given by $d$.

$Q$ and $\Gamma$ being finite and each step following a local rule,
one can easily simulate the run of a Turing machine in a cellular
automata. In our case, it is the role of the symbol $*$ which create,
after one iteration of the CA, a zone bounded by two walls containing
the representation of the empty tape $B_{\emptyset}$ and one symbol
$B_{\left(q_{i},\emptyset\right)}$ at the same position occupied
previously by $*$ representing the head of the Turing machine. The
walls move at speed $v>1$ in each direction, leaving $B_{\emptyset}$
on their way: in the absence of errors, the head moving at most at
speed 1 in a direction cannot meet a wall, and thus its run is not
affected by the finite nature of the tape.

Turing machines are the base tool to prove undecidability problems,
as the problem ``the machine reach the state $q_{\bot}$ in a finite
amount of steps'' is undecidable (or uncomputable): there is no algorithm
such that, given $(Q,q_{0},q_{\bot},\Gamma,\delta)$, can determine
if this machine halts in a finite amount of time. 

\subsection{Description of the CA}

\subsubsection{General description}

In the construction of \cite{BDPST15}, the maximum speed was $1$,
and the particles had speed $\frac{1}{4}$ and $\frac{1}{5}$. In
order to simplify the proof (e.g. the particles have integer speeds),
we choose the maximal speed to be $v=40$. The neighborhood radius
will then be also $v$. The alphabet $\A$ of cardinal $C<\infty$
is composed of the following symbols:
\begin{itemize}
\item $0$, which are spreading on the $B$ at speed $v$, but also on the
walls (if on the right side).
\item $B$, the tape on which the Turing machine is running. One can decompose
it in a finite number of $B_{\tilde{\gamma}}$, where $\tilde{\gamma}\in\Gamma\sqcup\left(Q\times\Gamma\right)\sqcup\Sigma$
where if $\tilde{\gamma}=\left(q,\gamma\right)\in Q\times\Gamma$,
$\gamma$ is the symbol written on the tape while $q$ is the state
of the head of the Turing machine which is positioned here. Otherwise
when $\tilde{\gamma}=\gamma\in\Gamma$ it is just the symbol written
on the tape. Finally $\Sigma$ is the set encapsulating the signals
used for the comparison: comparison signals $S_{1},S_{2}$, the destruction
signals and the position signal. If the state $q_{\bot}\in Q$ is
reached (when the machines halts), the symbol is replaced by a $0$,
which will spread on the tape around it.
\item $*$, which initializes the Turing machine and create walls on each
side.
\item Walls: the left and right inner walls with speed $\nicefrac{v}{5}$,
and the outer left and right walls with speed $\nicefrac{v}{4}$.
They stop the propagation of the $0$, but only in one direction;
they are erased if caught up by a $0$ from the other direction. If
the walls are created by a $*$ cell, the space between them is filled
with $B$.
\end{itemize}
\begin{figure}[h]
\centering{}\def\svgwidth{\columnwidth *4/4}
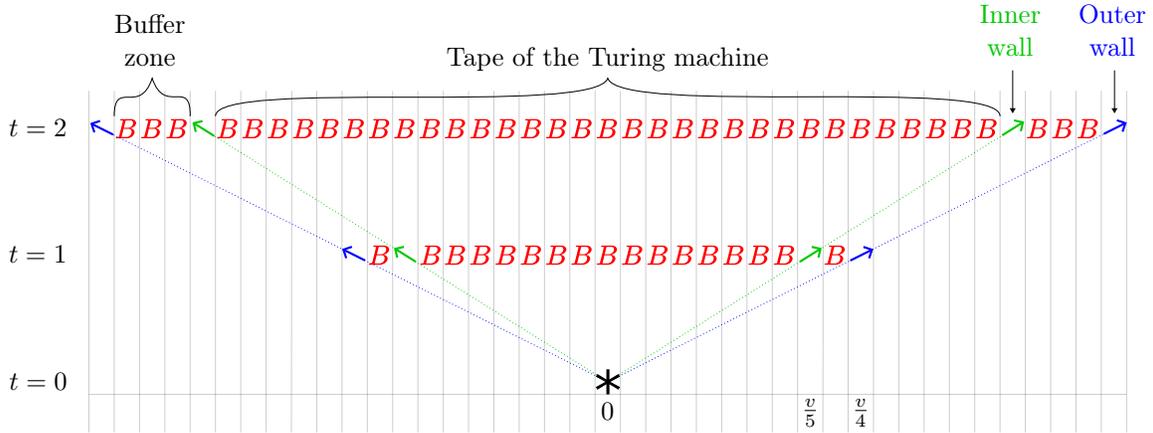\caption{Illustration of the behavior of the CA.}
\end{figure}

\begin{rem}
\label{rem:deltaB pas point fixe}In the following we consider that
the only fixed point of $F$ is $0^{\infty}$ (and thus $\delta_{0}$
is the only $F$-invariant Dirac mass). In the case that there is
a type of $B$ symbol such that $B^{\infty}$ is a fixed point, we
can define a new symbol $B^{\prime}$ with the same behavior under
the CA, but with the added rule that $f((B,\cdots,B))=B^{\prime}$
and $f((B^{\prime},\cdots,B^{\prime}))=B$.
\end{rem}

\subsubsection{Collisions}

The motivation behind taking this construction instead of a simpler
one is that \emph{in the absence of perturbation}, a cone created
by a $*$ symbol can't be erased. This external robustness is granted
by the double layer of walls each side of the cone and signals such
that, if two cones collides, only the youngest one (i.e. created by
the most recent $*$) survives. 

\begin{figure}[h]
\begin{centering}
\includegraphics{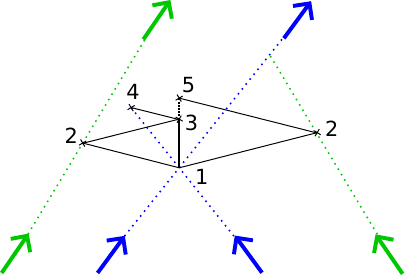}
\par\end{centering}
\caption{\label{fig:collision}Illustration of the collision process.}

\end{figure}

The collision is handled as illustrated in the space-time diagram
in Figure \ref{fig:collision}:
\begin{enumerate}
\item When two outer walls collides, they produce a vertical position signal,
as well as two comparison signals propagating trough the buffer zones.
\item They collide to their respective inner wall.
\item The comparison signal bouncing off the younger inner wall arrives
first to the position signal: a destruction signal is sent.
\item The destruction signal erases the older outer wall.
\item The other comparison signal is destroyed upon arrival, letting the
younger outer wall erasing the information in the older cone.
\end{enumerate}
All the comparison and destruction signals propagate at speed $v$
trough the buffer zones, to ensure that the older outer walls are
erased before encountering the younger inner wall. 

\subsubsection{Perturbation\label{subsec:Perturbation}}

At each step of the automaton, the configuration is perturbed by a
uniform noise of size $\epsilon>0$: independently from each other,
each cell has a probability $\epsilon$ to have its symbol replaced
by a symbol chosen uniformly in $\A$. The local rule is then $f_{\epsilon}\left(x_{\mathcal{N}},a\right)=\begin{cases}
1-\epsilon+\frac{\epsilon}{C} & \text{if }a=f(x_{\mathcal{N}})\\
\frac{\epsilon}{C} & \text{if }a\neq f(x_{\mathcal{N}})
\end{cases}.$ 

To make computations clearer, we define $\left(E_{i}^{t}\right)_{i,t\in\Z}$
independent random variables such that for all $i,t\in\Z$, $E_{i}^{t}(\Omega)=\left\{ \emptyset\right\} \bigcup\A\eqqcolon\A^{\prime}$
and $\begin{cases}
P\left(E_{i}^{t}=\emptyset\right)=1-\epsilon\\
P\left(E_{i}^{t}=a\right)=\frac{\epsilon}{C} & \forall a\in\A
\end{cases}.$ It defines a local rule $g:\A^{\mathcal{N}}\times\A^{\prime}\to\A$
with $g\left(x_{\mathcal{N}},e\right)=\begin{cases}
f\left(x_{\mathcal{N}}\right) & \text{if }e=\emptyset\\
e & \text{otherwise}
\end{cases}$, which can be made in a global rule $G:\A^{\Z}\times\A^{\prime^{\Z}}\to\A^{\Z}$
via $G\left(x,\left(e_{i}\right)_{i\in\Z}\right)_{j}=g\left(x_{j+\mathcal{N}},e_{j}\right)$.
It is not difficult to see that for all observable $A$, $F_{\epsilon}(x,A)=P\left(G\left(x,\left(E_{i}^{t}\right)_{i\in\Z}\right)\in A\right)$.
Thus, for a $F_{\epsilon}$-invariant measure $\pi_{\epsilon}$, we
can define a stationary sequence $\left(X^{t}\right)_{t\in\Z}$ which
verifies for all $t\in\Z$ the relation $X^{t+1}=G\left(X^{t},\left(E_{i}^{t+1}\right)_{i\in\Z}\right)$
and $X^{t}$ is distributed according to $\pi_{\epsilon}$. The event
$E_{i}^{t}=a\in\A$ is realized when an error is made at time $t$
in cell $i$, and a symbol $a$ is written instead of the expected
result of the CA.

Remark that because it only appear via errors, for each $F_{\epsilon}$-invariant
measure the probability to have a $*$ symbol in a given cell is $\frac{\epsilon}{C}$;
and by independence of the errors (i.e. the independence of $\left(E_{i}^{t}\right)$),
the probability to have at least one $*$ symbol over $n$ cells is
$1-\left(1-\frac{\epsilon}{C}\right)^{n}$. 

\subsection{Case where the TM doesn't halt}

Suppose that the Turing doesn't halt: let us show that for any collection
$\left(\pi_{\epsilon}\right)_{\epsilon>0}$ of $F_{\epsilon}$-invariant
measures, the value $\pi_{\epsilon}\left(\left[0\right]_{0}\right)$
does not converge towards 1 when $\epsilon$ goes to $0$. Here, we
will prove that there is a map $f:]0,1[\to\R$ such that $\pi_{\epsilon}\left([B]_{0}\right)\geq f(\epsilon)$
and $f(\epsilon)\tendsto{\epsilon}0l>0$.

A $*$ produces a zone of slope at least $\frac{v}{5}-1$ of $B$
symbols. Thus, in order to have a $B$, it suffices that $n$ step
before there was a $*$ symbol within the $2n\left(\frac{v}{5}-1\right)+1$
cells, and that there wasn't any error on the dependence cone of our
original cell over the last $n$ steps. The size of that cone is $\sum_{t=0}^{n-1}\left(2vt+1\right)=vn^{2}-n\left(v-1\right)$.
See Figure \ref{fig:nonHalt}.

\begin{figure}[H]
\def\svgwidth{\columnwidth *3/4}
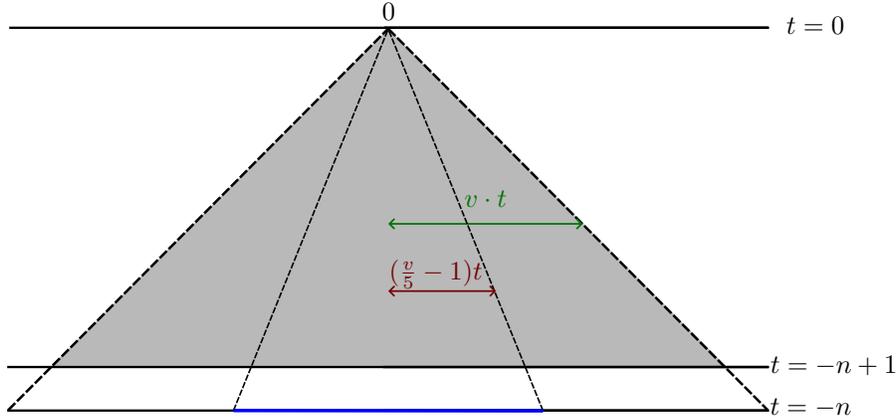\caption{\label{fig:nonHalt}Case where the TM doesn't halt. If there is no
error in the grayed area and a $*$ symbol in the blue area, there
must be a $B$ at $\left(0,0\right)$.}

\end{figure}

The $*$ being only able to appear via an error, these events are
disjoints for distinct $n$, thus
\begin{align*}
\pi_{\epsilon}\left([B]_{0}\right) & =P\left(X_{0}^{0}=B\right)\\
 & \geq\sum_{n=1}^{+\infty}P\left(\text{"a }*\text{ symbol at time }-n\text{"}\cap\text{"no error after in the dependence cone of \ensuremath{(0,0)}"}\right)\\
 & =\sum_{n=1}^{+\infty}P\left(\left(\bigcup_{i=-n\left(\frac{v}{5}-1\right)}^{n\left(\frac{v}{5}-1\right)}E_{i}^{-n}=*\right)\cap\left(\bigcap_{\substack{-n<t\leq0\\
-nv\leq i\leq nv
}
}E_{i}^{t}=\emptyset\right)\right)\\
 & =\sum_{n=1}^{+\infty}\left(1-\left(1-\frac{\epsilon}{C}\right)^{2n\left(\frac{v}{5}-1\right)+1}\right)\left(1-\epsilon\right)^{vn^{2}-n\left(v-1\right)}\\
 & \geq\sum_{n=1}^{+\infty}\left(1-\left(1-\frac{\epsilon}{C}\right)^{2n\left(\frac{v}{5}-1\right)}\right)\left(1-\epsilon\right)^{vn^{2}}\eqqcolon f(\epsilon).
\end{align*}
And we can conclude using Proposition \ref{prop:limiteaubord}, as
$f(\epsilon)\tendsto{\epsilon}0\frac{1}{C}\cdot\frac{\frac{v}{5}-1}{v}>0$.

\subsection{Case where the TM halts \label{subsec:casehalts}}

Suppose that the Turing machine halts after a finite amount of steps:
let us show that for any collection $\left(\pi_{\epsilon}\right)_{\epsilon>0}$
of $F_{\epsilon}$-invariant measures, the value $\pi_{\epsilon}\left(\left[0\right]_{0}\right)$
converges towards 1 when $\epsilon$ goes to $0$. Here, we will prove
that there that the probability to encounter any other symbol goes
to $0$.

\subsubsection{Definitions}

When the machine halts (in a finite time), the heads become a $0$
which spreads at speed $v$. Thus, in the absence of error, the space-time
zone filled with $B$ symbols between the two walls created by a $*$
is finite: denote by $T$ its height, and

\[
T_{j,t}\coloneqq\left\{ \left(i,s\right)\in\Z\times\Z\mid0<s-t\leq T,-\frac{v}{4}\left(s-t\right)\leq i-j\leq\frac{v}{4}\left(s-t\right)\right\} 
\]
the zone created by a $*$ in $j$ at time $t$. To bound the probability
to have each symbol in $0$ at time $0$, we use the stationary sequence
$\left(X^{t}\right)_{t\in\Z}$ with distribution $\pi_{\epsilon}$
and the description of errors $\left(E_{i}^{t}\right)_{i,t\in\Z}$
defined in Section \ref{subsec:Perturbation} to search the source
of the symbol. Loosely, we say that the symbol $a$ comes from a symbol
$b$ in $j\in\Z$ at time $t\in\Z^{-}$ an error makes a $b$ appear
here and there is no $*$ symbol ``between''. As the speed of propagation
of $B$ and walls is bounded by $\frac{v}{4}$, the zone of space-time
we consider to be $*$-free is the parallelogram $P_{j,t}\coloneqq\left\{ \left(i,s\right)\in\Z\times\Z\mid t<s\leq0,-\frac{v}{4}\left(s-t\right)\leq i-j\leq\frac{v}{4}\left(s-t\right),-\frac{v}{4}\left|s\right|\leq i\leq\frac{v}{4}\left|s\right|\right\} $.
Therefore a symbol $a$ comes from a symbol $b$ in $j\in\Z$ at time
$t\in\Z^{-}$ if $E_{j}^{t}=b$ and $\forall(i,s)\in P_{j,t},E_{i}^{s}\neq*$.
\begin{rem*}[Area of a parallelogram]

\begin{figure}[h]
\centering{}\def\svgwidth{\columnwidth *3/4}
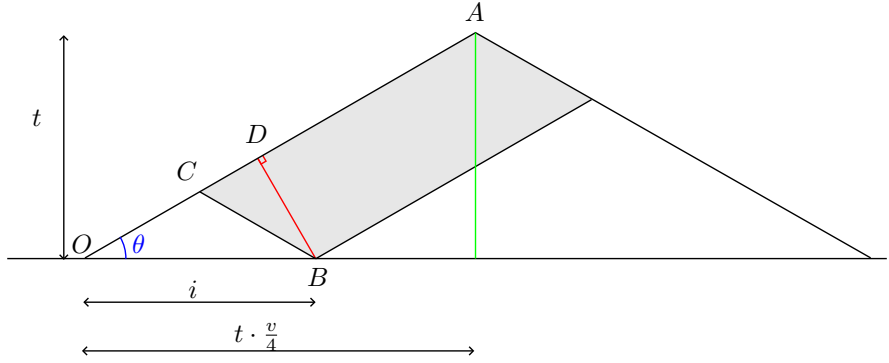\caption{\label{fig:AreaCone}Area of a parallelogram.}
\end{figure}
Suppose an error occurred at time $t$ in $j\leq0$ (point $B$),
the case $j>0$ being symmetrical. Define $i=\frac{v}{4}t+j$. In
order to say that a symbol in $0$ at time $0$ (point $A$) comes
from this error, we can suppose the parallelogram $P_{j,t}$ described
in Figure \ref{fig:AreaCone} to be $*$-free. Using its notations,
the area of this parallelogram has formula:

\begin{align*}
\text{Area} & =BD\cdot AC\\
 & =OB\sin\left(\theta\right)\left(AO-OC\right)\\
 & =i\sin\left(\arctan\left(\nicefrac{4}{v}\right)\right)\left(t\sqrt{1+\left(\nicefrac{v}{4}\right)^{2}}-\frac{i}{2}\sqrt{1+\left(\nicefrac{4}{v}\right)^{2}}\right)\\
 & \geq it\cdot\underset{\eqqcolon K^{\prime}>0}{\underbrace{\sin\left(\arctan\left(\nicefrac{4}{v}\right)\right)\left(\sqrt{1+\left(\nicefrac{v}{4}\right)^{2}}-\frac{v}{8}\sqrt{1+\left(\nicefrac{4}{v}\right)^{2}}\right)}}.
\end{align*}
\end{rem*}

\subsubsection{The $B$ symbols}

A $B$ symbol in $0$ at time $0$ must come either from an error
in $0$ at time $0$ (denote this event by $\Omega_{0}$), or from
a $*$ symbol in the last $T$ steps (denoted by $\Omega_{1}$) or
further (denoted by $\Omega_{2}$), or from at least two simultaneous
errors (denoted by $\Omega_{3}$) that can spread $B$. Thus, $\left\{ X_{0}^{0}=B\right\} \subset\Omega_{0}\cup\Omega_{1}\cup\Omega_{2}\cup\Omega_{3}$
and 
\begin{itemize}
\item $P\left(\Omega_{0}\right)=P\left(E_{0}^{0}=B\right)=\frac{\epsilon}{C}\tendsto{\epsilon}00$.
\item $P\left(\Omega_{1}\right)=P\left(\bigcup_{\substack{1\leq t\leq T\\
-\frac{v}{4}t\leq i\leq\frac{v}{4}t
}
}E_{i}^{-t}=*\right)\leq\sum_{t=1}^{T}\sum_{i=-\frac{v}{4}t}^{\frac{v}{4}t}P\left(E_{i}^{-t}=*\right)\leq\left(\frac{v}{2}+1\right)T^{2}\frac{\epsilon}{C}\tendsto{\epsilon}00.$
\item $\Omega_{2}$: the Turing machine (or the tape) must be perturbed
before it halts, therefore 
\begin{align*}
P\left(\Omega_{2}\right) & \leq\sum_{t\geq T}\sum_{j=-\frac{v}{4}t+1}^{\frac{v}{4}t-1}P\left(\text{"a }*\text{ at }j\text{ at time \ensuremath{-t}"}\cap\text{"an error in the finite zone"}\cap\text{"no }*\text{ in the parallelogram after"}\right)\\
 & \leq\sum_{t\geq0}\sum_{j=-\frac{v}{4}t+1}^{\frac{v}{4}t-1}P\left(\left(E_{j}^{-t}=*\right)\cap\left(\bigcup_{(k,s)\in T_{j,-t}}E_{k}^{s}\neq\emptyset\right)\cap\left(\bigcap_{(k,s)\in P_{j,-t}\backslash T_{j,-t}}E_{k}^{s}\neq*\right)\right)\\
 & \leq\sum_{t\geq0}2\sum_{i=1}^{\frac{v}{4}t+1}\frac{\epsilon}{C}\left(1-\left(1-\epsilon\right)^{\frac{v}{4}T^{2}}\right)\left(1-\frac{\epsilon}{C}\right)^{K^{\prime}it-\frac{v}{4}T^{2}}\\
 & =2\frac{\epsilon}{C}\left(1-\left(1-\epsilon\right)^{\frac{v}{4}T^{2}}\right)\left(1-\frac{\epsilon}{C}\right)^{-\frac{v}{4}T^{2}}\sum_{t\geq0}\sum_{i=1}^{\frac{v}{4}t+1}\left(1-\frac{\epsilon}{C}\right)^{K^{\prime}it}
\end{align*}
One can rewrite $\sum_{t\geq0}\sum_{i=1}^{\frac{v}{4}t+1}x^{it}=\sum_{n}a_{n}x^{n}$
with $a_{n}=\left|\left\{ \left(i,t\right)\mid i\leq\frac{v}{4}t+1,\;it=n\right\} \right|\leq d(n)\leq2\sqrt{n}.$
But $\sum_{n}\sqrt{n}x^{n}\underset{x\to1^{-}}{\sim}\frac{\sqrt{\pi}}{2\left(1-x\right)^{3/2}}$,
so 
\begin{align*}
P\left(\Omega_{2}\right) & \leq2\frac{\epsilon}{C}\left(1-\left(1-\epsilon\right)^{\frac{v}{4}T^{2}}\right)\left(1-\frac{\epsilon}{C}\right)^{-\frac{v}{4}T^{2}}\sum_{n\geq0}\sqrt{n}\left(1-\frac{\epsilon}{C}\right)^{K^{\prime}n}\\
 & \underset{\epsilon\to0}{\sim}\epsilon^{2}\cdot\frac{K^{\prime\prime}}{\epsilon^{\nicefrac{3}{2}}}\\
 & \tendsto{\epsilon}00.
\end{align*}
\item $\Omega_{3}$: to spread the $B$ symbol, it must be protected in
a similar way a $*$ do with walls; there must be at least two simultaneous
errors to create it.
\begin{align*}
P\left(\Omega_{3}\right) & \leq\sum_{t\geq0}P\left(\text{"two errors at time }-t\text{"}\cap\text{"no }*\text{ in the parallelograms after"}\right)\\
 & \leq\sum_{t\geq0}4\sum_{a=0}^{\frac{v}{4}t+1}\sum_{b=a}^{\frac{v}{4}t+1}P\left(\left(E_{a-\frac{v}{4}t}^{-t}\neq\emptyset\right)\cap\left(E_{b-\frac{v}{4}t}^{-t}\neq\emptyset\right)\cap\left(\bigcap_{(k,s)\in A_{a-\frac{v}{4}t,-t}\cup A_{b-\frac{v}{4}t,-t}}E^{s}\neq*\right)\right)\\
 & \leq\sum_{t\geq0}4\sum_{a=0}^{\frac{v}{4}t+1}\sum_{b=a}^{\frac{v}{4}t+1}\epsilon^{2}\left(1-\frac{\epsilon}{C}\right)^{K^{\prime}bt}\\
 & =4\epsilon^{2}\sum_{t\geq0}\sum_{a=0}^{\frac{v}{4}t+1}\sum_{b=a}^{\frac{v}{4}t+1}\left(1-\frac{\epsilon}{C}\right)^{K^{\prime}bt}
\end{align*}
One can rewrite $\sum_{t\geq0}\sum_{a=0}^{\frac{v}{4}t+1}\sum_{b=a}^{\frac{v}{4}t+1}\left(1-\frac{\epsilon}{C}\right)^{K^{\prime}bt}=\sum_{n}b_{n}x^{n}$
with $x=\left(1-\frac{\epsilon}{C}\right)^{K^{\prime}}$ and $b_{n}=\left|\left\{ \left(a,b,t\right)\mid a\leq b\leq\frac{v}{4}t+1,\;bt=n\right\} \right|$.
Remark that
\[
b_{n}=\sum_{\substack{b\mid n\\
b\leq\frac{1+\sqrt{1+vn}}{2}
}
}b\leq\sqrt{vn}d(n).
\]
We know that (see for example \cite{NR83}) $d(n)\leq n^{\frac{2\ln(2)}{\ln(\ln n)}}$.
Define $N_{0}\in\N$ such that for all $n\geq N_{0}$, $d(n)\leq n^{\frac{1}{6}}$
and therefore $b_{n}\leq\sqrt{v}n^{\nicefrac{2}{3}}$. Using Proposition
\ref{prop:equivalentaubord}, we have $\sum_{n\geq0}n^{\nicefrac{2}{3}}x^{n}\underset{x\to1^{-}}{\sim}\frac{\Gamma(\nicefrac{5}{3})}{\left(1-x\right)^{\nicefrac{5}{3}}}$.
Similarly as the previous point, one can then conclude that, 
\[
P\left(\Omega_{3}\right)\leq\underset{\tendsto{\epsilon}00}{\underbrace{4\epsilon^{2}\sum_{n=0}^{N_{0}}b_{n}}}+\underset{\sim\frac{4\epsilon^{2}\sqrt{v}\Gamma\left(\nicefrac{5}{3}\right)}{\left(K^{\prime}\frac{\epsilon}{C}\right)^{\nicefrac{5}{3}}}\tendsto{\epsilon}00}{\underbrace{4\epsilon^{2}\sqrt{v}\sum_{n\geq N_{0}}n^{\nicefrac{2}{3}}\left(1-\frac{\epsilon}{C}\right)^{K^{\prime}n}}}\tendsto{\epsilon}00.
\]
\end{itemize}
Finally, $P\left(X_{0}^{0}=B\right)=\pi_{\epsilon}\left(\left[B\right]_{0}\right)\tendsto{\epsilon}00$.

\subsubsection{Walls}

By symmetry, it is enough to only work on the walls going towards
the right. Regardless if they are inner or outer, they move only if
they have a $B$ symbol to their left (and they carry it with them
next). A wall can then only come from two sources: a $*$, or an error
making the wall appear, but then must be adjacent to a $B$ to survive
a step. For a $\nearrow$ with speed $c>0$, $\left\{ X_{0}^{0}=\nearrow\right\} \subset\Omega_{0}^{\prime}\cup\Omega_{1}\cup\Omega_{2}\cup\Omega_{3}^{\prime}$,
where $\Omega_{1}$ and $\Omega_{2}$ are the same as before, $\Omega_{0}^{\prime}$
is the event that there is an error at $(0,0)$ producing a $\nearrow$,
and $\Omega_{3}^{\prime}$ the event that it comes from an error producing
a $\nearrow$ in the past, thus needing an adjacent $B$ to survive.
\begin{itemize}
\item $P\left(\Omega_{o}^{\prime}\right)=P\left(E_{0}^{0}=\nearrow\right)=\frac{\epsilon}{C}\tendsto{\epsilon}00$.
\item $P\left(\Omega_{1}\cup\Omega_{2}\right)\tendsto{\epsilon}00$ by previous
calculations (a $*$ produces more $B$ symbols than walls).
\item Finally,
\begin{align*}
P\left(\Omega_{3}^{\prime}\right) & \leq\sum_{t=1}^{+\infty}P\left(\text{an error produces a }\nearrow\text{ at time }-t\text{ in }-tc\text{ with a }B\text{ symbol left of it and no }*\text{ after}\right)\\
 & =\sum_{t=1}^{+\infty}P\left(\left(E_{-tc}^{-t}=\nearrow\right)\cap\left(X_{-tc-1}^{-t}=B\right)\cap\left(\bigcap_{-t<s\leq0}E_{-sc}^{-s}\neq*\right)\right)\\
 & =\sum_{t=1}^{+\infty}\frac{\epsilon}{C}\pi_{\epsilon}\left(\left[B\right]_{0}\right)\left(1-\frac{\epsilon}{C}\right)^{t}\\
 & \leq\pi_{\epsilon}\left(\left[B\right]_{0}\right)\tendsto{\epsilon}00.
\end{align*}
\end{itemize}

\section{Remaining questions}
\begin{itemize}
\item In the last construction, we don't detail what are the stochastically
stable measures in the case when the Turing machine doesn't halts;
in particular, if we ignore remark \ref{rem:deltaB pas point fixe},
we don't know if $\delta_{B}$ is stable under this perturbation.
A potential approach would be to consider that the $0$ spreads into
the $B$, and the walls let the $B$ spreads into the $0$ in some
sense. It may be linked with the $3$-state cyclic cellular automaton,
as defined in \cite{HB19}: the $2$ spreads into the $1$, which
spreads into the $0$, which spreads into the $2$. In the cited paper,
the authors show a formula for the limit measure depending on the
measure behind the starting configuration. We conjecture that for
this CA is strongly stochastically stable under a uniform perturbation,
with the stable measure being $\frac{1}{3}\left(\delta_{0}+\delta_{1}+\delta_{2}\right)$.
In our construction, it may lead to $\alpha\delta_{0}+\left(1-\alpha\right)\delta_{B}$
(with a parameter $\alpha$ left to be determined) being the only
stable measure.
\item In this article we only showed example of cases where there is only
one stochastically stable measure. Define $\M_{0}$ to be the set
of all stochastically stable measure (for an $F$ and its perturbation
$\left(F_{\epsilon}\right)_{\epsilon>0}$ fixed). What would be its
properties, and can we characterize the sets $M$ that can be constructed
as a $\M_{0}$ ? Ongoing research tends to show that for a continuous
perturbation, $\M_{0}$ is at least connected. Using similar constructions
as the one in the final section of this article and in \cite{HS18},
it seems that such sets could be characterized by their computational
properties.
\end{itemize}
\bibliographystyle{plain}
\bibliography{biblio}

\appendix

\section{\label{sec:Demonstrations-lemmas}Other demonstrations of section
\ref{sec:lemmas}}

\begin{proof}[Proof of Proposition \ref{lem:abelien}]
(Adapted from \cite{Gourdon20}) Suppose $\frac{a_{n}}{b_{n}}\to l$.
Let $\epsilon>0$, $n_{0}\in\N$ be such that $\forall n\geq n_{0},\:\left|\frac{a_{n}}{b_{n}}-l\right|\leq\epsilon$,
i.e. $\left|a_{n}-l\cdot b_{n}\right|\leq\epsilon b_{n}$. Then $\forall x\in[0,1[$,
\[
\left|\sum_{n\geq n_{0}}\left(a_{n}-l\cdot b_{n}\right)x^{n}\right|\leq\epsilon\sum_{n\geq n_{0}}b_{n}x^{n}
\]
and thus
\[
\left|\sum_{n\geq0}\left(a_{n}-l\cdot b_{n}\right)x^{n}\right|\leq\sum_{n=0}^{n_{0}-1}\left|a_{n}-l\cdot b_{n}\right|x^{n}+\epsilon\sum_{n\geq n_{0}}b_{n}x^{n}.
\]
As $\sum b_{n}$ is a divergent series with positive terms, $\sum_{n\geq n_{0}}b_{n}x^{n}\tendsto x{1^{-}}+\infty$.
Therefore $\exists\lambda_{\epsilon}<1,\forall x\in[\lambda_{\epsilon},1[,$
\[
\epsilon\sum_{n\geq n_{0}}b_{n}x^{n}\geq\sum_{n=0}^{n_{0}-1}\left|a_{n}-l\cdot b_{n}\right|.
\]
Then for $x\in[\lambda_{\epsilon},1[$, 
\[
\left|\sum_{n\geq0}\left(a_{n}-l\cdot b_{n}\right)x^{n}\right|\leq2\epsilon\sum_{n\geq n_{0}}b_{n}x^{n}\leq2\epsilon\sum_{n\geq0}b_{n}x^{n}
\]
and finally
\[
\left|\frac{\sum_{n\geq n_{0}}a_{n}x^{n}}{\sum_{n\geq n_{0}}b_{n}x^{n}}-l\right|\leq2\epsilon.
\]
\end{proof}
\begin{proof}[Proof of Corollary \ref{cor:abelien}]
By Cauchy product, $\sum A_{n}x^{n}=\left(\sum a_{n}x^{n}\right)\left(\sum x^{n}\right)=\frac{\sum a_{n}x^{n}}{1-x}$.
Then, 
\[
\frac{\sum^{\infty}a_{n}x^{n}}{\sum^{\infty}b_{n}x^{n}}=\frac{\sum^{\infty}A_{n}x^{n}}{\sum^{\infty}B_{n}x^{n}}.
\]
One have the result using the Lemma \ref{lem:abelien}, as $B_{n}>0$
and $\sum B_{n}$ diverges too.
\end{proof}

\end{document}